\documentclass[10pt, a4paper]{article}

\usepackage{a4wide, amsfonts,enumerate,amsmath,amsfonts,xspace,array,delarray,theorem,amssymb,epsfig}
\usepackage{graphicx,color}
\usepackage[latin1]{inputenc}%

\theoremstyle{break}\theorembodyfont{\it}
\theoremheaderfont{\normalfont\bfseries}

\newtheorem{theo}{Theorem}

\newtheorem{theonumbered}{Theorem}

\newtheorem{lem}[theo]{Lemma}

\newtheorem{prop}[theo]{Proposition}

\newtheorem{rem}[theo]{Remark}

\newenvironment{proof}{\noindent{\bf Proof: }}
                {\leavevmode\unskip\nobreak\hskip2em plus1fill
                $\scriptstyle\square$\vskip\theorempostskipamount\par}

\let\lt=<
\let\gt=>
\def\freccia{{\longrightarrow}}
\def\R{{\mathbb R}}

\def\N{{\mathbb N}}

\def\SSS{ {\mathcal {S}}_{p,q}}
\let\phi=\varphi
\let\eps=\varepsilon

\def\be{\begin{equation}}
\def\ee{\end{equation}}
\def\beq{\begin{eqnarray*}}
\def\eeq{\end{eqnarray*}}

\def\d{{\rm d}}
\newcommand{\fer}[1]{(\ref{#1})}

\def\hg{\widehat g}

\begin{document}

\title{Strong Convergence towards self-similarity for \\
one-dimensional dissipative Maxwell models}

\author{G. Furioli \thanks{University of Bergamo, viale
Marconi 5, 24044 Dalmine, Italy. \texttt{giulia.furioli@unibg.it}},
 A. Pulvirenti \thanks{Department\
of Mathematics, University of Pavia, via Ferrata 1, 27100 Pavia,
Italy. \texttt{ada.pulvirenti@unipv.it}}, E. Terraneo
\thanks{Department of Mathematics, University of Milano, via Saldini
50, 20133 Milano, Italy. \texttt{Elide.Terraneo@mat.unimi.it}}, and
G. Toscani\thanks{Department of Mathematics, University of Pavia,
via Ferrata 1, 27100 Pavia, Italy.
\texttt{giuseppe.toscani@unipv.it}}
 }

\maketitle

\abstract{We prove the propagation of regularity, uniformly in
time, for the scaled solutions of the one-dimensional dissipative
Maxwell models introduced in \cite{BBLR}. This result together
with the weak convergence towards the
stationary state proven in \cite{PT06} implies the strong
convergence in Sobolev norms and in the $L^1$ norm towards it
depending on the regularity of the initial data. As a consequence,
the original non scaled solutions are also proved to be convergent
in $L^1$ towards the corresponding self-similar homogenous cooling
state. The  proof is based on the (uniform in time) control of the
tails of the Fourier transform of the solution, and it holds for a
large range of values of the mixing parameters. In particular, in
the case of the one-dimensional inelastic Boltzmann equation, the
result does not depend of the degree of inelasticity. This
generalizes a recent result of Carlen, Carrillo and Carvalho
\cite{CCC08}, in which, for weak inelasticity, propagation of
regularity for the scaled inelastic Boltzmann equation was found
by means of a precise control of the growth of the Fisher
information.}

\section{Introduction}
In 2003 Ben-Avraham and coworkers \cite{BBLR} introduced a
one-dimensional model of the Boltzmann equation, in which binary
collision processes are given by arbitrary linear collision rules
 \be\label{coll}
 v^* = pv+qw, \quad w^* = qv +pw ; \quad p\geq q>0.
 \ee
The positive constants $p$ and $q$ represent the mixing
parameters, namely the portion of the pre--collisional velocities
$(v,w)$ which generate the post--collisional ones $(v^*,w^*)$.
Under the hypothesis of \emph{constant} collision frequency, this
mechanism of collision leads to the integro-differential equation
of Boltzmann type,
 \be \partial_t f(v,t) =
\int_{\R}\left( \frac 1J f(v_*,t)
 f(w_*,t) - f(v,t) f(w,t)\right)\, \d w \label{eq:boltz}
  \ee
where now $(v_*,w_*)$ are the pre-collisional velocities that
generate the couple $(v,w)$ after the interaction and $J = p^2-q^2$ is the Jacobian of the transformation
of $(v,w)$ into $(v^*,w^*)$.
 As observed in \cite{BBLR}, while in the long-time limit velocity
distributions are generically self-similar, there is a wide spectrum
of possible behaviors. The velocity distributions are characterized
by algebraic or stretched exponential tails and the corresponding
exponents depend sensitively on the collision parameters.
Interestingly, when there is energy or momentum conservation, the
behavior is universal.

Since the integrals $\int_\R v^n\, f(v,t)\, \d v$,\ $n\ge 0$,  obey a closed
hierarchy of equations \cite{BK},  moments can be evaluated
recursively, starting from mass conservation. In particular,
choosing as initial density a normalized probability density
$f_0$ satisfying
 \be\label{norm1}
 f_0\geq 0,\quad \int_{\R} f_0(v)\, \d v =1 \, , \quad    \int_{\R} v f_0(v)\, \d v =  0
 \, ,
 \quad  \int_{\R} v^2 f_0(v)\, \d v = 1,
 \ee
it follows that both mass and momentum are preserved in time, while
the second moment varies according to the law
 \be\label{2-th}
 E(t) = \int_{\R} v^2 f(v,t)  \,\d v = \exp\left\{ (p^2+q^2-1)t\right\}.
 \ee
Special cases include the elastic model ($p^2 + q^2 = 1$), which is
the analogous of the well-known Kac model \cite{Kac, McK}, the
inelastic collisions ($p + q = 1$ ) \cite{BMP}, the granules model
($p+q < 1$ ) \cite{RBL},  the inelastic Lorenz gas ($q = 0, p<1$)
\cite{MP99}, and in addition  energy producing models ($p> 1$)
\cite{TK}.

By \fer{2-th} it follows that the second moment of the solution is
not conserved, unless the collision parameters satisfy
 \[
 p^2 + q^2 = 1.
 \]
If this is not the case, the energy can grow to infinity or decrease
to zero, depending on the sign of $p^2+q^2-1$. In both cases,
however,  stationary solutions of finite energy do not exist, and
the large--time behavior of the system can at best be described by
self-similarity properties. The standard way to look for
self-similarity is to scale the solution according to the rule
 \[
  g(v,t) = \sqrt{ E(t)}f\left( v\sqrt{ E(t)}, t \right) .
  \]
This scaling implies that $\int_\R v^2\, g(v,t)\, \d v= 1$ for all $t \ge 0$.

The large-time behavior of the density $f(v,t)$ has been studied in
\cite{PT06}, by resorting to the Fourier transform version of the
Boltzmann equation \fer{eq:boltz}, which reads
\be
\label{four}
\partial_t \hat f(\xi,t)= \hat f(p\xi,t) \hat f(q\xi,t)-\hat f(\xi,t).
\ee
which convertes for the scaled solution $g(t)$ into the following
\be\label{g-eq}
\partial_t \hat g(\xi,t)+\frac
12\left(p^2+q^2-1\right)\xi\partial_\xi\hat g(\xi,t)=
 \hat g(p\xi,t) \hat g(p\xi,t)-\hat g(\xi,t).
\ee
It is worth  recognizing from equation \eqref{four} an equivalent formulation of equation
\eqref{eq:boltz}
making use of the convolution operator
\[
\partial_t f(v,t)= f_p\ast f_q(v,t)-f(v,t)
\]
where we used the shorthand
 \[
f_p(v) = \frac 1p f\left( \frac vp \right).
 \]
The key argument for studying equations \eqref{four} and \eqref{g-eq} is the use of a metric for probability
densities with finite and equal moments of order $[\alpha]$, where, as usual,
$[\alpha]$ denotes the entire part of the real number $\alpha$:
\begin{equation}
\label{ds}
d_\alpha (f,g) = \sup_{\xi \in {\R}} \frac{|\widehat{f}(\xi) -
\widehat{g}(\xi)|}{|\xi|^\alpha}.
\end{equation}
The metric \fer{ds} has been introduced in~\cite{GTW} to investigate the trend to
equilibrium of the solutions to the Boltzmann equation for Maxwell molecules.
Further applications of $d_\alpha$ can be found in~\cite{CGT, PT, GJT, TV}.

The study of the time evolution of the $d_\alpha$-metric, with $\alpha = 2 +
\delta$, for some suitable $0<\delta <1$, enlightened the range of the mixing
parameters for which one can expect that the scaled function $g(t)$ converges
(weakly) towards a steady profile $g_\infty$ at an exponential rate. More
precisely, let us define, for fixed $ p$ and $q$ the function
 \be
\label{key}
 \SSS (\delta) = p^{2+\delta} + q^{2+\delta} -1 -\frac{2+\delta}2\left(p^2+q^2
 -1\right).
  \ee
 Then, the sign of $\SSS$ determines the asymptotic behavior of the
distance  $d_\alpha(g_1(t), g_2(t))$ between two scaled solutions
$g_1(t)$, $g_2(t)$ issued from two initial data $f_{1,0}$,
$f_{2,0}$. In particular, it has been proved in \cite{PT06} that
if there exists  $\tilde \delta \in (0,1)$ such that
 $\SSS (\delta) <0$ for $0< \delta < \tilde \delta$, we can conclude that
$d_{2+\delta}(g_1(t), g_2(t))$ converges exponentially to zero if
initially finite according to the bound
 \[
  d_{2+\delta} (g_1(t), g_2(t)) \leq \exp\left\{-\left| \SSS(\delta)\right|t\right\}d_{2+\delta}(f_{1,0},f_{2,0}),
\quad \delta \in (0,\tilde \delta).
\]
It has been also proved in \cite{PT06} that in this case
a unique steady state $g_\infty$ exists for equation \eqref{g-eq} and for any initial data
$f_0$ with $(2+\tilde \delta)$ finite moments, we have for
$0<\delta<\tilde \delta$:
 \be
\label{dec1}
 d_{2+\delta} (g(t),g_\infty) \leq \exp\left\{-\left| \SSS(\delta)\right|t\right\}d_{2+\delta}(f_0, g_\infty) \to 0,
 \quad t\to +\infty.
\ee
On the original non scaled solution $f(t)$, the limit behavior corresponds to a self-similar state
$f_\infty(v,t)= \frac 1{\sqrt{E(t)}} g_\infty \left(\frac v{\sqrt{E(t)}}\right)$.
Note that, by construction, $\SSS (0)= 0$, and thus
$\min_{\delta \in (0,1)}\{\SSS \}\leq 0$. A numerical evaluation of the region
where the minimum of the function $\SSS$ is negative for $p,q \in
[0,2]$ is reported in \cite{PT06}. This region includes the
relevant cases of both inelastic ($p+q=1$) and elastic ($p^2 +q^2
=1$) collisions, as well as the case, among others, in which
$p=q=1$.

Despite the fact that the large-time behavior of the solution to the Boltzmann--like
equation \eqref{eq:boltz} can be described in terms of the $d_\alpha$-metric, which is
equivalent to the weak$^*$ convergence of measures \cite{CT07}, the strong
convergence of the scaled density $g(t)$ towards $g_\infty$ has never been
proved before.

A similar problem occurs in dissipative kinetic theory, where the weak
convergence of the (scaled) solution to the inelastic Boltzmann equation for
Maxwell molecules towards the homogeneous cooling state with polynomial tails is known
to hold \cite{BCT} in the $d_\alpha$-metric framework, but the strong
convergence is still unknown in full generality. A recent paper by Carlen,
Carrillo and Carvalho \cite{CCC08} shows that in some cases one can prove that
the strong convergence holds. Their result, however, requires a small
inelasticity regime, which in our setting of the mixing parameters means
$p+q=1$, and at the same time $1-p^2-q^2 << 1$. In this case, in fact,  one can
resort to methods close to the elastic situation, in which the (controlled)
growth  of  the Fisher information, coupled with the exponential decay of the
$d_\alpha$-metric allows to prove the uniform propagation of Sobolev regularity, and
from this, by interpolation, the strong convergence.

Propagation of Sobolev regularity for both Kac equation and the elastic Boltzmann
equation for Maxwellian molecules, together with the precise exponential rate of the
strong convergence to
the Maxwellian equilibrium $M$ has been proved in \cite{CGT}. The
advantage of working with the classical elastic Boltzmann equation
relies on the fact that one can resort to the $H$-theorem.
A careful reading of
\cite{CGT}, however, allows to conclude that the proof of uniform propagation of regularity
makes use of
the following condition
for the distance between the solution $f(t)$ and the Maxwellian
$M$:
 \be
\label{ok}
 \sup_{\xi\in\R} |\hat f(\xi,t) - \hat M(\xi)| \to 0,\quad t\to +\infty.
 \ee
 While the convergence of $H(f)(t)$ to $H(M)$ implies the $L^1$--convergence and therefore \fer{ok},
the same condition continues to hold provided the
Fourier transform
of the solution to the (scaled) Boltzmann equation $g(t)$ and of the stationary state $g_\infty$ satisfy
\fer{dec1} together with a suitable (uniform) decay at infinity in the $\xi$ variable, of
the type
 \be
\label{tail}
 {(1 + \kappa|\xi|)^\mu}|\hat g(\xi,t)| \le K,\quad \xi\in\R
 \ee
 for some positive constants $\kappa,\ K$ and $\mu$. In fact, in
 this case, for any given $R>0$,
 \[
 |\hat g(\xi,t) - \hat g_\infty(\xi)| \le d_\alpha(g(t), g_\infty) R^\alpha +
 \frac {2K}{(\kappa R)^\mu},
 \]
which implies, optimizing over $R$,
 \[
|\hat g(\xi,t) - \hat g_\infty(\xi)| \le C(\alpha, \mu,  \kappa, K)
d_\alpha(g(t), g_\infty)^{\mu/(\alpha+\mu)},\quad \xi\in\R,\ t\geq 0.
 \]
Therefore, in presence of condition \fer{tail}, the decay to zero of the
$d_\alpha$-metric implies the decay to zero of $\sup_{\xi\in\R}|\hat g(\xi,t) - \hat
g_\infty(\xi)|$ for $t\to +\infty$. We will remark in Section \ref{sec-strong} that condition \fer{tail} on the initial density
$f_0$ holds provided the square root of $f_0$ has some regularity.
For example,
the Fisher information of $f_0$, which controls the $H^1$-norm of the square
root, controls $|\xi||\hat f_0(\xi)|$ (cfr. the proof in \cite{LT}).

Condition \fer{tail} is difficult to prove directly from the
equation satisfied by the scaled density $g(t)$, due to the
presence of the drift term in equation \eqref{g-eq}.
To simplify the proof of the various bounds we will introduce a semi-implicit
discretization of equation \fer{g-eq}, that is, for a small time interval
$\Delta t$, we will consider the solution to
\begin{equation}
\label{d-eq}
\frac{ \hat g(\xi,t+\Delta t)- \hat g(\xi,t)}{\Delta t} +\frac
12\left(p^2+q^2-1\right)\xi\partial_\xi\hat g(\xi,t+ \Delta t)=
 \hat g(p\xi,t) \hat g(q\xi,t)-\hat g(\xi,t).
 \end{equation}
Inspired by the integral formulation of the stationary state $g_\infty$ in \cite{BC03},  the solution $\hat
g(\xi,t+\Delta t)$  to \fer{d-eq} at time $t+ \Delta t$ is shown to
be a convex combination of the probability densities $\hat g(\xi,t)$
and $\hat g(p\xi,t) \hat g(q\xi,t)$. Precisely, for $p^2+q^2<1$ we have
\be
\label{sol-d}
\hat g(\xi, t+\Delta t)=\frac{r}{\Delta t}\
\int_{1}^{+\infty}\left(\Delta t \,\, \hat g(\tau p\xi,t) \hat
g(\tau q\xi,t) \ +(1-\Delta t)\ \hat g( \tau \xi,t) \right)\frac{\d
\tau}{\tau^{\frac r{\Delta t}+1}},
\ee
where
 \[
\frac 1r=\frac{1-p^2-q^2}{2}
 \]
 and for $p^2+q^2 >1$ we have
 \be
\label{sol-non diss}
\hat g(\xi, t+\Delta t)=-\frac{r}{\Delta t}\
\int_{0}^{1}\left(\Delta t \,\, \hat g(\tau p\xi,t) \hat
g(\tau q\xi,t) \ +(1-\Delta t)\ \hat g( \tau \xi,t) \right)\frac{\d
\tau}{\tau^{\frac r{\Delta t}+1}}.
\ee
Eventhough the solution of the semi-implicit discretization  is in both dissipative
and non dissipative case a uniform approximation of the scaled solution $g(t)$ (see Proposition \ref{prop-converg} in Section \ref{approssim}), we have been able to exploit this approximation only in the dissipative case.
In this case,
the integral formulation is also useful to recover
regularity properties of the steady solution $g_\infty$. In
particular, we shall prove that the steady state has the Gevrey
regularity of class $\lambda$, namely
 \[
 e^{\mu|\xi|^\lambda}|\hat g_\infty(\xi)| \le 1, \qquad |\xi| > \rho,
 \]
where the constants $\lambda, \mu$ and $\rho$ are related both to $p, q$
and to the number of moments which are bounded initially. Our main
results are summarized into the following statements.
In what follows we will denote $g_0$ the initial data of a scaled solution $g(t)$, even if
of course $g_0=f_0$, the initial data of the original, non scaled solution $f(t)$.

\begin{theo} \label{teo-bounds}
 Assume $0<q\leq p$ satisfying $p^2+q^2<1$ and such that there is $\tilde\delta \in (0,1)$ for which  $\SSS(\delta) <0$, for $0<\delta<\tilde\delta$.
 Let $g(t)$ be the weak
solution of the equation (\ref{g-eq}), corresponding to the
initial density $g_0$ satisfying the normalization conditions
\fer{norm1}, and
$$
\int_{\R}|v|^{2+\tilde\delta}\, g_0(v)\, \d v<+\infty.
$$
If in addition
 \begin{equation}\label{con3}
|\hat g_0(\xi)|\leq \frac 1{(1+\beta|\xi|)^\nu}, \quad  |\xi|>R,
\end{equation}
for some $R>0$, $\nu >0$ and $\beta>0$, then there exist
$\rho>0$,\ $k >0$,\ $\beta' >0$, $\nu'>0$
 such that
$g(t)$ satisfies
 \begin{equation}
\label{bbs}
 |\hat g(\xi, t)| \leq
\left\{
\begin{aligned}
 &\frac 1{1+ k\xi^2},\quad |\xi| \leq \rho ,\quad t\geq 0\\
& \frac 1{(1+ \beta ' |\xi|)^{{\nu}'}},\quad |\xi|> \rho,\quad
t\geq 0.
\end{aligned}
\right .
\end{equation}
\end{theo}
Theorem \ref{teo-bounds} is proven in Section
\ref{regularity}. The second result is concerned with the regularity of the
steady state $g_\infty$ and is proven in Section \ref{Smoothness}.

\begin{theo}
  \label{thm.smoothness}
  Assume $0<q\leq p$ satisfying $p^2+q^2 <1$ and such that there exists  $\tilde \delta \in (0,1)$
  for which
 $\SSS (\delta) <0$ for $0< \delta < \tilde \delta$
 so that a non-trivial steady state $g_\infty$ to the Boltzmann equation \eqref{g-eq} exists.
Let us denote $\lambda \in (0,2)$ the exponent such that $p^\lambda +q^\lambda =1$.
  Then $g_\infty$ is a smooth function and belongs to the $\lambda$-th Gevrey class $G^\lambda(\R)$, i.e.
  \begin{align*}
    \big| \hat g_\infty(\xi)\big| &\leq \exp\big(-\mu|\xi|^\lambda\big), \quad |\xi|>\rho
  \end{align*}
with suitable positive numbers $\rho$ and $\mu$.
\end{theo}

As a byproduct of both proofs of Theorem \ref{teo-bounds} and Theorem \ref{thm.smoothness}, we also get
the uniform propagation of the Gevrey regularity for solutions $g(t)$ issued from Gevrey initial data.

\begin{theo}\label{gev-prop}
 Assume $0<q\leq p$ satisfying $p^2+q^2 <1$ and such that there exists  $\tilde \delta \in (0,1)$
 for which
 $\SSS (\delta) <0$ for $0< \delta < \tilde \delta$
and let us denote $\lambda \in (0,2)$ the exponent such that
$p^\lambda +q^\lambda =1$. Let $g(t)$ be the weak solution of the
equation (\ref{g-eq}), corresponding to the initial density $g_0$
satisfying the normalization conditions \fer{norm1}, and
$$
\int_{\R}|v|^{2+\tilde \delta}\, g_0(v)\, \d v<+\infty.
$$
If in addition
 \begin{equation*}
|\hat g_0(\xi)|\leq e^{-\beta |\xi|^\nu}, \quad  |\xi|>R,
\end{equation*}
for some $R>0$, $\nu >0$ and $\beta>0$, then there exist $\rho>0$ and $\kappa >0$
 such that
$g(t)$ satisfies
 \begin{equation*}
 |\hat g(\xi, t)| \leq
\left\{
\begin{aligned}
& e^{-\kappa \xi^2},\quad |\xi| \leq \rho,\quad t\geq 0\\
& e^{-\kappa |\xi|^{\min (\nu, \lambda)}} ,\quad
|\xi|> \rho,\quad t\geq 0.
\end{aligned}
\right .
\end{equation*}
\end{theo}
The previous results are enough to prove the convergence in strong norms
towards the steady state $g_\infty$. In consequence of both Theorems
\ref{teo-bounds} and \ref{thm.smoothness},  we can show in fact the uniform in time propagation of
regularity in Sobolev spaces of high degree
 \[
 \Vert g \Vert_{\dot H^\eta(\R)}^2 = \int_\R |\xi|^{2\eta} |\hat g(\xi)|^2 \, \d \xi
 \]
with $\eta > 0$. It is enough to apply the technique developed in \cite{CGT} for
the Boltzmann equation for Maxwell molecules for showing that whenever the
equation propagates a tiny degree of regularity, as in Theorem \ref{teo-bounds},
this implies that the equation
propagates regularity of any degree.
Then, using the regularity in high
Sobolev spaces, we can pass from the weak convergence in $d_\alpha$-metric
obtained in \cite{PT06} into convergence in all Sobolev norms, and strong $L^1$
convergence at an explicit exponential rate for a certain class of initial
data. This is the objective of Section \ref{sec-strong} and the main result is
summarized as follows.

\begin{theo} \label{teo-L1}
Assume $0<q\leq p$ satisfying $p^2+q^2 <1$ and such that there
exists  $\tilde \delta \in (0,1)$ for which $\SSS (\delta) <0$
for $0< \delta < \tilde \delta$
 and let  $g_\infty$ be the unique stationary solution of
\eqref{g-eq}.
 Let the initial
density $g_0$ satisfy  the normalization conditions
\fer{norm1}, and
$$
\int_{\R}|v|^{2+\tilde\delta}\, g_0(v)\, \d v<+\infty.
$$
If in addition $g_0 \in {H^\eta(\R)}$ for some $\eta > 0$,
$\sqrt{g_0} \in {\dot H^{\nu}(\R)}$ for some $\nu > 0$, then the
solution $g(t)$ of (\ref{g-eq}) converges strongly in $L^1$ with
an exponential rate towards the stationary solution $g_\infty$,
i.e., there exist positive constants $C$ and $\gamma$ explicitly
computable such that
 \[
 \Vert g(t) - g_\infty \Vert_{L^1(\R)} \le C e^{-\gamma t},\quad t\geq 0.
 \]
\end{theo}
Thanks to the scaling invariance of the $L^1$ norm, Theorem \ref{teo-L1} allows to
deduce also the strong convergence of the original non scaled solution $f(t)$ to the self-similar state
$f_\infty (v, t)=\frac 1{\sqrt{E(t)}} g_\infty \left( \frac v {\sqrt{E(t)}}\right )$:
\[
 \Vert f(t) - f_\infty(t) \Vert_{L^1(\R)} \le C e^{-\gamma t},\quad t\geq 0.
\]

Finally, in Section \ref{lyap} we will discuss in details the relevant case in which $p+q =1$, which
corresponds to the one-dimensional inelastic Boltzmann equation for Maxwell
molecules \cite{BMP}. In this case, in fact, it is known that an explicit
stationary solution to equation \fer{g-eq} exists,
 \[
 \hat g_\infty(|\xi|) = (1 + |\xi|) e^{-|\xi|},
 \]
or, in the physical space
 \[
 g_\infty(v) = \frac 2{\pi(1+v^2)^2}.
 \]
This stationary solution is independent of the values of $p$ and $q$, and
within the set of functions $f$ which satisfy the normalization conditions
\fer{norm1}, is the minimum of the convex functional
 \[
 H(f) = -\int_\R \sqrt{f(v)} \, dv.
 \]
 This suggests the idea that $H$ is an entropy functional for the scaled
 equation \fer{g-eq}, but the proof of this conjecture would require to satisfy
 an inequality which is the analogous of the Shannon entropy power inequality
  \cite{Bla, Sta}, we are not able to prove.

It is interesting to remark that the ideas of the present paper, which are
valid for the $p+q =1$ case, can be fruitfully used for the three-dimensional
dissipative Boltzmann equation for Maxwell molecules, extending the validity of
the recent analysis of \cite{CCC08} to a general coefficient of restitution.
However, the proof of the analogous of Theorem \ref{teo-bounds} to the
three-dimensional case requires heavy computations, we will publish separately
elsewhere.


\section{Preliminary results} \label{Par-Tos}

Let us consider the one-dimensional kinetic models of
Maxwell-Boltzmann type:
\begin{equation}\label{eq1}
\left\{
\begin{aligned}
&\partial_t f(v,t)= \int_{\R} \left (\frac{1}{J}f(v_*, t) f(w_*,
t) -f(v,t)f(w,t)\right )
\, \d w\\
& f(v,0)= f_0(v)
\end{aligned}
\right .
\end{equation}
where $f(v,t) : \R \times \R^+ \freccia \R$  denotes the
distribution of particles with velocity $v\in\R$ at the time
$t\geq 0$ and $(v_*, w_*)$  are the pre-collisional velocities
that generate post-collisional $(v,w)$:
\[
\left\{
\begin{aligned}
& v= pv_*+qw_*\\
& w = qv_*+pw_*
\end{aligned}
\right .
\]
and $0 <q \leq p$.

Moreover let us suppose  $f_0$ satisfying the normalization
conditions \fer{norm1}. In order to avoid the presence of the
Jacobian  we write equation \eqref{eq1} in weak form, namely
\[
\frac{\d }{\d t} \int_{\R}\Phi(v) f(v,t)\, \d v= \iint_{\R^2} f(v,
t) f(w, t)\left ( \Phi(v^*)-\Phi(v)\right )\, \d w\, \d v,
\]
and
\[
 \lim_{t\to 0}\int_{\R}\Phi(v) f(v,t)\, \d v=
\int_{\R}\Phi(v) f_0(v)\, \d v
\]
for any $\Phi$ bounded and continuous on $\R$. At least formally
by choosing $\Phi(v)=1$ and  $\Phi(v)=v$ one shows that, under
conditions \fer{norm1} both the mass and momentum are preserved.
By choosing $\Phi(v)=v^2$ we obtain that the energy of the
solution
$$
E(t)=\int_{\R}v^2f(v,t) \ \d v
$$
satisfies the equality
$$
E(t)={\rm e}^{(p^2+q^2-1)t}E(0).
$$
 Therefore, unless $p^2 +q^2 =1$, the energy is not preserved. In the
dissipative case $p+q=1$, one has additionally that the moment is
\emph{always} preserved, while the energy is decreasing.

 Following Bobylev \cite{Bob}, the weak
form of  equation \eqref{eq1} is equivalent to the equation in the
Fourier variables:
\[
\left\{
\begin{aligned}
&\partial_t \hat f(\xi,t)= \hat f(p\xi,t) \hat f(q\xi,t)-\hat f(\xi,t),\quad t>0\\
& \hat f(\xi,0)= \hat f_0(\xi).
\end{aligned}
\right .
\]
The existence and uniqueness of a solution for any initial data
$f_0$ satisfying  (\ref{norm1}) can be established in the same way
as for the elastic Kac equation.

\begin{theo}[Theorem of existence and uniqueness \cite{PT06, PT}]\label{esist}
We consider  $f_0$ satisfying the normalization conditions
\fer{norm1} and the following Cauchy problem: \be\label{cauchy}
\left\{
\begin{aligned}
&\partial_t \hat f(\xi,t)=  \hat f( p\xi,t) \hat f( q\xi ,t)-\hat
f(\xi,t),
 \quad t>0\\
& \hat f(\xi, 0)= \hat f_0(\xi).
\end{aligned}
\right . \ee Then, there exists a unique nonnegative solution
$f\in C^1 \left( [0, +\infty), L^1(\R)\right )$ to equation
(\ref{cauchy}) satisfying for all $t>0$:
\[
\int_{\R} f(v, t) \, \d v =1,\quad \int_{\R} f(v,t)\, v\,\d v=0.
\]
\end{theo}
In order to investigate some properties of the solution is useful
to introduce the following rescaling
$$
\hat g(\xi,t)=\hat f\left(\frac{\xi}{\sqrt{E(t)}},t\right).
$$
The function $\hat g(t)$ preserves the energy and it satisfies the
equation:
\begin{equation}\label{geq1}
\left\{
\begin{aligned}
&\partial_t \hat g(\xi,t)+\frac
12\left(p^2+q^2-1\right)\xi\partial_\xi\hat g(\xi,t)=
 \hat g(p\xi,t) \hat g(q\xi,t)-\hat g(\xi,t),\quad t>0\\
& \hat g(\xi,0) =\hat f_0(\xi):= \hat g_0(\xi).
\end{aligned}
\right .
\end{equation}
In the relevant case $p+q=1$, equation \fer{geq1} admits an
explicit stationary state \cite{BMP}
$$
\hat
 g_{\infty}(\xi)=(1+|\xi|){\rm
 e}^{-|\xi|}.
 $$
Pareschi and Toscani proved in \cite{PT06} that in a quite large
range of values of the mixing para\-me\-ters, a  unique stationary
state exists and the unique weak solution $g(t)$ converges to
$g_{\infty}$. In all the cases $p+q\neq 1$ (excluding $p^2 + q^2
=1$) the stationary state is not explicit, even if some properties
can be extracted from the analysis of the evolution equation. The
convergence takes place in a Fourier distance introduced in
\cite{GTW} to investigate the trend to equilibrium of the solution
to the Boltzmann equation for Maxwellian molecules. We recall the
definition of this distance. Let  $0\leq \delta<1$ and let
\[
M_{2+\delta} =\left\{ f\geq 0\, :\  \int_\R f(v)\, \d v=1, \
\int_\R v\, f(v) \, \d v =0,\ \int_\R v^2\, f(v)\, \d v=1,\
\int_\R |v|^{2+\delta}\,  f(v)\, \d v < \infty, \right \}.
\]
We introduce on ${M}_{2+\delta}$ the distance:
$$
d_{2+\delta}(f,g)=\sup_{\xi\in \R}\frac{|\hat f(\xi)-\hat
g(\xi)|}{|\xi|^{2+\delta}}.
$$
Let us define, for $\delta \ge 0$,
 \[
 \SSS (\delta) = p^{2+\delta} + q^{2+\delta} -1 -\frac{2+\delta}2\left(p^2+q^2
 -1\right).
 \]
The result of Pareschi and Toscani is as follows.
\begin{theo}[Pareschi--Toscani \cite{PT06}]\label{teo-pt}
Assume $0<q\leq p$ and such that there exists  $\tilde \delta \in
(0,1)$ for which $\SSS (\delta) <0$ for $0< \delta < \tilde
\delta$. Let $g(t)$ be the weak solution of  equation
(\ref{geq1}), corresponding to the initial density $g_0$
satisfying the normalization conditions \eqref{norm1} and
$$
\int_{\R}|v|^{2+\tilde \delta}\, g_0(v)\, \d v<+\infty.
$$
Then $g(t)$ satisfies for $0<\delta \leq  \tilde \delta$ and for
$c_\delta >0$:
$$
\int_{\R}|v|^{2+\delta}\, g(v,t)\, \d v\leq c_\delta, \quad t\geq
0.
$$
Moreover, there exists a unique stationary state $g_\infty$ to
equation (\ref{geq1}) which satisfies for $0<\delta <\tilde
\delta$:
\[
\int_{\R}|v|^{2+\delta}\, g_\infty(v)\, \d v<+\infty,
\]
$g(t)$ converges exponentially fast in Fourier metric towards
$g_\infty$ and the following bound holds for $0<\delta <\tilde
\delta$:
$$
d_{2+\delta}(g(t),g_\infty)\leq {\rm e}^{-|{\cal
S}_{p,q}(\delta)|t}d_{2+\delta}(g_0,g_\infty).
$$
\end{theo}


\section{An iteration process}\label{approssim}
 The goal of this section is to build up a sequence of functions
 $\left\{g^N(\xi,t)\right \}$ which approximates uniformely the solution $\hat g(\xi,t)$.
 In order to do this, for any fixed $T>0$ we consider firstly a semi-implicit discretization in time of equation
 (\ref{geq1}) by  partitioning  the  interval  $[0,T]$ into $N$
 subintervals  and we define thus
 the approximate solution at any time $t=j\frac{T}{N}$ for $j=0,\dots,N$.
Secondly, we define $g^N(\xi,t)$ on the whole interval $[0,T]$ by
 interpolation and lastly we show the convergence of the approximation to the solution.

In order to lighten the reading, we will postpone almost all the
proofs of this technical section at the end of the paper in
Appendix \ref{annexe}.

\subsection*{The approximate equation}
In this paragraph  we define the approximate solution at any time
$t=j\frac{T}{N}$ for $j=0,\dots,N$ by an iteration process and
study some of its properties.

Let $T>0$  and $\Delta t=\frac {T}{N}$ for $N\in \N$, $N>T$. Let
 $\hat\phi_{j}^N(\xi)$, $j=0,\dots, N$ be the sequence:
\begin{equation}\label{eqdis}
\left\{
\begin{aligned}
&\hat\phi_0^N(\xi)=\hat g_0(\xi)\\
&\frac{\hat\phi_{j+1}^N(\xi)-\hat\phi_j^N(\xi)}{\Delta
t}=\frac1{r} \xi\ \frac{\d}{\d \xi}{\hat
\phi_{j+1}}^N(\xi)+\hat\phi_j^N(p\xi)\hat\phi_j^N(q\xi)-\hat\phi_j^N(\xi),
\quad j=0,\dots, N-1
\end{aligned}
\right .
\end{equation}
where $\frac 1r=\frac{1-p^2-q^2}{2}$.
\begin{prop} \label{intform}
Assume $0<q\leq p$.
If $g_0$ verifies the normalization conditions \eqref{norm1},
then  there exists a unique sequence of bounded function
  $\hat\phi_{j}^N$ for $j=1,\dots,N$ satisfying (\ref{eqdis}).
\end{prop}
\begin{proof}
Let us begin by proving that  $\hat\phi_{1}^N$ is well defined.
 In a similar way as
in \cite{BC03} we multiply equation (\ref{eqdis}) by
$\left(-\frac{r}{\Delta t}\right)\ \mathrm{sgn\xspace} \ \xi \
|\xi|^{-\frac{r}{\Delta t}-1}$ and  obtain
$$
\frac{\d}{\d \xi}\left(\hat \phi_{1}^N(\xi) |\xi|^{-\frac r{\Delta
t}}\right)=\left(-\frac{r}{\Delta t}\right) \ \mathrm{sgn\xspace}
\ \xi \ |\xi|^{-\frac{r}{\Delta t}-1}\left(\Delta t\
\hat\phi_0^N(p\xi)\hat\phi_0^N(q\xi)+(1-\Delta t)\ \hat
\phi_0^N(\xi)\right).
$$
We assume now $p^2+q^2<1$. For any positive $\xi$ we integrate on
$[\xi,+\infty)$ and since $\hat\phi_0^N(\xi) $ is bounded we get:
$$
\hat\phi_{1}^N(\xi)|\xi|^{-\frac r{\Delta t}}=\frac{r}{\Delta
t}\int_\xi^{+\infty}\left(\Delta t\
\hat\phi_0^N(ps)\hat\phi_0^N(qs)+(1-\Delta t)\ \hat
\phi_0^N(s)\right) s^{-\frac r{\Delta t}-1}\d s.
$$
Finally by the change of variables $\tau=s/\xi$ we are led to:
\begin{equation}\label{inteq}
\hat\phi_{1}^N(\xi)=\frac{r}{\Delta t}\
\int_{1}^{+\infty}\left(\Delta t \
\hat\phi_0^N(p\tau\xi)\hat\phi_0^N(q\tau\xi)+(1-\Delta t)\ \hat
\phi_0^N(\tau\xi)\right)\frac{\d \tau}{\tau^{\frac r{\Delta
t}+1}}.
\end{equation}
For any negative $\xi$ we integrate on $(-\infty,\xi]$ and in a
similar way we obtain that equality (\ref{inteq}) holds for any
$\xi\neq 0$. Moreover since  $g_0$ satisfies the conditions (3)
then $\hat\phi_{0}^N=\hat g_0$  belongs  to ${\cal C}^1(\R)$ and
\begin{equation}\label{gzero}
\hat g_0(0)=1\ \ \ \left|\frac{\d\,\hat g_0(\xi)}{\d\,\xi}
\right|\leq 1\ \ \ \frac{\d\,\hat g_0}{\d\,\xi} (0)=0.
\end{equation}
Therefore  the function $\hat\phi_{1}^N$ can be defined by
continuity in $\xi=0$ and it is the unique, bounded and ${\cal
C}^1(\R)$ solution of (\ref{eqdis}).  By an iteration argument the
same conclusion holds for any $\hat\phi_{j}^N$ obtaining for
$j=0,\dots, N-1$
\[
\hat\phi_{j+1}^N(\xi)=\frac{r}{\Delta t}\
\int_{1}^{+\infty}\left(\Delta t \
\hat\phi_j^N(p\tau\xi)\hat\phi_j^N(q\tau\xi)+(1-\Delta t)\ \hat
\phi_j^N(\tau\xi)\right)\frac{\d \tau}{\tau^{\frac r{\Delta
t}+1}}.
\]
For $p^2+q^2 >1$, we repeat the same argument by integrating on
$[0,\xi]$ and $[\xi, 0]$ and we get in the end
\[
\hat\phi_{j+1}^N(\xi)=-\frac{r}{\Delta t}\
\int_{0}^{1}\left(\Delta t \
\hat\phi_j^N(p\tau\xi)\hat\phi_j^N(q\tau\xi)+(1-\Delta t)\ \hat
\phi_j^N(\tau\xi)\right)\frac{\d \tau}{\tau^{\frac r{\Delta
t}+1}}.
\]
\end{proof}
Applying Fubini's theorem  we can remark that for $p^2+q^2 <1$ any
$\hat\phi_{j+1}^N(\xi)$ is the Fourier transform of
$\phi_{j+1}^N(v)$ where for $j=0,\dots, N-1$ \be
\label{succ-spazio} \left\{
\begin{aligned}
&\phi_0^N(v)= g_0(v)\\
& \phi_{j+1}^N(v)=\frac{r}{\Delta t}\
\int_{1}^{+\infty}\left(\Delta t \ \frac 1 \tau
\left(\phi_{j,p}^N\ast \phi_{j, q}^N\right ) \left(\frac v
\tau\right )+(1-\Delta t)\ \frac 1 \tau \phi_j^N\left(\frac
v\tau\right )\right)\frac{\d \tau}{\tau^{\frac r{\Delta t}+1}},
\end{aligned}
\right . \ee with $\phi_{j,p}^N(v)= \frac 1p \phi_j^N\left(\frac
vp\right )$ and similarly for $\phi_{j,q}^N$. Analogously, for
$p^2+q^2 >1$, $\hat\phi_{j+1}^N(\xi)$ is the Fourier transform of
$\phi_{j+1}^N(v)$ where for $j=0,\dots, N-1$
\be\label{fourier-nondiss} \left\{
\begin{aligned}
&\phi_0^N(v)= g_0(v)\\
& \phi_{j+1}^N(v)=-\frac{r}{\Delta t}\ \int_{0}^{1}\left(\Delta t
\ \frac 1 \tau \left(\phi_{j,p}^N\ast \phi_{j, q}^N\right )
\left(\frac v \tau\right )+(1-\Delta t)\ \frac 1 \tau
\phi_j^N\left(\frac v\tau\right )\right)\frac{\d \tau}{\tau^{\frac
r{\Delta t}+1}}.
\end{aligned}
\right . \ee

\medskip
In what follows, the function $\SSS(\delta)$ is defined as in
\eqref{key}.

\begin{prop}\label{prop-momenti}
Assume $0<q\leq p$ such that there exists $\tilde \delta \in
(0,1)$ for which $\SSS (\delta) <0$ for $0< \delta <\tilde
\delta$.
Let $\phi_j^N$, for $j=0,\dots, N$ defined as in
\eqref{succ-spazio} or \eqref{fourier-nondiss}, with $g_0$
satisfying the normalization conditions \eqref{norm1} and $\int_\R
 |v|^{2+ \tilde \delta} \, g_0(v)\, \d v <+\infty$.
Then, for $0<\delta <\tilde \delta$ there exists $C_\delta >0$
such that for $N$ large enough (depending on $\delta$, $T$, $p$
and $q$), for $j=0,\dots, N$ we get \be \label{momenti}
\phi_j^N(v) \geq 0, \ \int_\R \phi_j^N(v)\, \d v=1,\, \ \int_\R
v\, \phi_j^N(v) \, \d v=0,\ \int_\R v^2 \, \phi_j^N(v)\, \d v=1,\
\int_\R |v|^{2+\delta} \, \phi_j^N(v)\, \d v \leq C_\delta. \ee
\end{prop}

\begin{rem} The equalities in (\ref{momenti}) imply that there exists
$C>0$ such that
\begin{equation}
\label{mom} \left| \hat\phi_j^N(\xi)\right|\leq 1,\ \ \ \left|
\frac{\d\hat\phi_j^N(\xi)}{\d \xi}\right|\leq C\ \ \ {\rm and}\ \
\ \left| \frac{\d^2\hat\phi_j^N(\xi)}{\d \xi^2}\right|\leq 1\ \ \
\end{equation}
 for any $ \xi\in \R\ $, for  $N$ large enough   and  for
$j=0,\dots,N$. The first two inequalities have already been proved
in the proof of  Proposition \ref{intform}.

\end{rem}
\subsection*{Definition of the sequence $\left\{g^N(\xi,t)\right \}_N$}
We are now in position to define the sequence
$\left\{g^N(\xi,t)\right \}_N$. We define $g^N(\xi,j\frac
TN)=\hat\phi_j^N(\xi)$ for $j=0,\dots,N$. We extend the definition
on the whole interval by interpolation. More precisely, let us
define:
\[
g^N(\xi,t)=
\begin{cases}
\hat g_0(\xi)&t=0\\
\alpha(t)\hat\phi^N_{K_N-1}(\xi)+(1-\alpha(t))\hat\phi^N_{K_N}(\xi)&
0<t\leq T
\end{cases}
\]
where for  $0<t\leq T$ we have $(K_N-1)\frac TN<t\leq K_N \frac
TN$ for  $K_N\in \{1,\dots,N\}$ and more precisely there is a
function $0\leq \alpha(t)<1$ such that $t=\alpha(t)(K_N-1)\frac
TN+(1-\alpha(t))K_N \frac TN$. Any $g^N(\xi,t)$ is continuous on
$\R\times[0,T]$ and for any $t\in [0,T]$ it belongs to ${\cal
C}^2(\R)$.

\bigskip

The result of convergence is therefore as follows.
\begin{prop}\label{prop-converg}
There is a subsequence $\left\{g^{N_l}(\xi,t)\right \}_l$ of
$\left\{g^N(\xi,t)\right \}_N$  which converges uniformely on any
compact set of\ $\R\times[0,T]$ to the solution $\hat g(\xi,t)$.
\end{prop}


\section{Propagation of regularity}\label{regularity}
 In this section we prove Theorem \ref{teo-bounds}.
Thanks to the uniform convergence of a subsequence of the
approximate solutions $g^{N}(\xi,t) $ to the solution $\hat
g(\xi,t)$ and to the definition of $g^{N}(\xi,t)$, it is enough to
prove the bounds \eqref{bbs}  for any $\hat \phi_j^N(\xi) $
uniformly for $N\in \N$ and $j=0,\dots, N$. The control of low
frequences is a direct consequence of the properties \eqref{norm1}
of the initial data and of the convergence to zero of the distance
$d_{2+\delta}(g(t),g_\infty)$ and  it is proven in the following
lemma.

\begin{lem}\label{close}
Assume $0<q\leq p$ satisfying $p^2+q^2<1$ and such that there is
$\tilde\delta \in (0,1)$ for which  $\SSS(\delta) <0$ for
$0<\delta<\tilde\delta$. Let $g(t)$ be the weak solution of the
equation (\ref{g-eq}), corresponding to the initial density $g_0$
satisfying the normalization conditions \fer{norm1}, and
$$
\int_{\R}|v|^{2+\tilde\delta}\, g_0(v)\, \d v<+\infty.
$$
Let $\hat\phi_j^N$ the approximation defined in \eqref{eqdis}. For
any $0<k<\frac 12$ there exists $\rho>0$ such that for any fixed
$T >0$ and any $N\in \N$ large enough
we get \be\label{est-sol} |\hat g(\xi,t)|\leq \frac
1{1+k\xi^2},\quad |\xi|\leq \rho, \quad t\geq 0, \ee
\be\label{est-succ} |\hat \phi_j^N(\xi)|\leq \frac 1{1+k\xi^2},
\quad |\xi|\leq \rho,\quad j=0,\dots, N. \ee
\end{lem}

\begin{proof}
Let us begin by proving \eqref{est-sol}. We remark first that
since
 for all $t\geq 0$
\[
 g(v,t)\geq 0,\quad \int_\R g(v,t)\, \d v=1, \quad \int_\R v\, g(v,t)
\, \d v=0,\quad \int_\R v^2 \, g(v,t)\, \d v=1
\]
we can deduce at once
\[
\hat g(\xi,t) = 1-\frac {\xi^2} 2 + o_t(\xi^2),\quad \xi \to 0
\]
and so for  $t\geq 0$ and  $0<k<\frac 12$ there exists $\rho =
\rho (t) >0$ such that
\[
|\hat g(\xi, t)| \leq \frac 1{1+ k\xi^2},\quad |\xi| \leq \rho(t)
\]
but this is not enough because we need an estimate independent of
$t$. We have therefore to proceed differently. By Theorem
\ref{teo-pt}, for any fixed  $0<\delta<\tilde\delta$ we have
\[
d_{2+\delta} (g(t), g_\infty) = \sup_{\xi\in\R} \frac {|\hat
g(\xi, t) -\hat g_\infty(\xi)|}{|\xi|^{2+\delta}} \leq e^{-|
\SSS(\delta)|t} d_{2+\delta}  (g_0, g_\infty), \quad t\geq 0.
\]
Moreover, since $g_\infty\geq 0$, $\int_\R g_\infty(v)\, \d v=1$,
$\int_\R v\, g_\infty(v)\, \d v=0$, $\int_\R v^2\, g_\infty(v)\,
\d v=0$ we have
\[
\hat g_{\infty} (\xi)  = 1-\frac {\xi^2}2 + o(\xi^2),\quad \xi\to
0,
\]
so we get uniformly for $t\geq 0$ and $\xi\neq 0$:
\[
\begin{aligned}
|\hat g(\xi, t)| &\leq \left|\hat g_\infty(\xi)\right| + \frac {|\hat g(\xi, t) - \hat g_\infty(\xi)|}{|\xi|^{2+\delta} }|\xi|^{2+\delta}\\
&\leq 1-\frac {\xi^2}2 + o(\xi^2) + C e^{-|\SSS(\delta)| t}
|\xi|^{2+\delta}, \quad \xi\to 0.
\end{aligned}
\]
This shows that for any $0<k <\frac 12$ there exists $\rho >0$
such that for all $t\geq 0$
\[
|\hat g(\xi, t)| \leq 1 -k\xi^2, \quad |\xi| \leq \rho
\]
or, which is equivalent, for any $0<k<\frac 12$ there exists $\rho
>0$ such that for all $t\geq 0$
\[
|\hat g(\xi, t)| \leq \frac 1{1+ k \xi^2},\quad |\xi| \leq \rho.
\]
In order to prove \eqref{est-succ}, we would like to exploit again
the $d_{2+\delta}$ distance. For $N\in\N$ and $j=1, \dots, N-1$
let us estimate first $d_{2+\delta} ( \phi_{j+1}^N, \hat
g_\infty)$. We recall that \be \label{eqstaz} 0=\frac 1r\xi
\frac{\d}{\d\xi}\hat g_\infty(\xi)+\hat g_\infty(p\xi)\hat
g_\infty(q\xi)-\hat g_\infty(\xi) \ee with $\frac 1r = \frac
{1-p^2-q^2}2$. By considering the two equations  (\ref{eqdis}) and
(\ref{eqstaz}) we have
\[
\begin{aligned}
\frac{\hat \phi_{j+1}^N (\xi) -\hat g_\infty(\xi)-(\hat \phi_{j}^N
(\xi) -\hat g_\infty(\xi))}{\Delta t}&=\frac 1r\xi
\left(\frac{\d}{\d\xi}\hat \phi_{j+1}^N (\xi)-\frac{\d}{\d\xi}\hat
g_\infty(\xi)\right)\\
&+\hat \phi_j^N( p \xi) \hat \phi_j^N(q\xi) -\hat
g_\infty(p\xi)\hat g_\infty(q\xi)-\left(\hat \phi_j^N( \xi)-\hat
g_\infty(\xi)\right).
\end{aligned}
\]
In the same way as in Proposition \ref{intform} we obtain the
 integral form:
\[
\begin{aligned}
\hat \phi_{j+1}^N (\xi) -\hat g_\infty(\xi)&=\frac{r}{\Delta
t}\int_1^{+\infty}\Delta t\left(\hat \phi_{j}^N (p\tau \xi)\hat
\phi_{j}^N (q\tau \xi)-\hat g_\infty(p\tau\xi)\hat
g_\infty(q\tau\xi)\right)\\
&+(1-\Delta t)\left(\hat \phi_{ j}^N(\tau \xi)-\hat
g_\infty(\tau\xi) \right)\frac{d\tau}{\tau^{\frac{r}{\Delta
t}+1}}.\\
\end{aligned}
\]
Therefore for $\xi\neq 0$ we get:
\[
\begin{aligned}
&\frac {|\hat \phi_{j+1}^N (\xi) -\hat
g_\infty(\xi)|}{|\xi|^{2+\delta}}
\\
&\leq \frac r{\Delta t}\int_1^{+\infty} \left(\Delta t \frac {|
\hat \phi_j^N( p\tau \xi) \hat \phi_j^N(q\tau\xi) - \hat
g_\infty(p\tau\xi)
 \hat g_\infty(q\tau\xi)|} {|\xi|^{2+\delta}}
+(1-\Delta t) \frac {|\hat \phi_j^N(\tau \xi) -\hat g_\infty
(\tau\xi)|}{|\xi|^{2+\delta}}\right )  \frac {\d \tau}{\tau^{\frac
r{\Delta t}+1}}.
\end{aligned}
\]
Now for $\tau \neq 0$
\[
 \frac {| \hat \phi_j^N( p\tau \xi) \hat \phi_j^N(q\tau\xi) - \hat g_\infty(p\tau\xi)
 \hat g_\infty (q\tau\xi)|}{|\xi|^{2+\delta}} \leq
d_{2+\delta}(\phi_j^N,  g_\infty) (p^{2+\delta} + q^{2+\delta})
\tau^{2+\delta}
\]
and so
\[
\begin{aligned}
&\frac {|\hat \phi_{j+1}^N (\xi) -\hat
g_\infty(\xi)|}{|\xi|^{2+\delta}} \leq \frac r{\Delta t}
d_{2+\delta}(\phi_j^N, g_\infty ) \left(\int_1^{+\infty}
\left(\Delta t (p^{2+\delta} + q^{2+\delta}) + 1-\Delta t\right ) \tau^{2+\delta}\,   \frac {\d \tau}{\tau^{\frac r{\Delta t}+1}}\right ) \\
&\leq   d_{2+\delta}(\phi_j^N,  g_\infty)\, \frac r{\Delta t} \,
(\Delta t (p^{2+\delta} + q^{2+\delta} -1) + 1) \,
\int_1^{+\infty} \frac {\d \tau} { \tau^{\frac r{\Delta
t}-1-\delta}}.
\end{aligned}
\]
For $N > \frac {T(2+\delta)}{r}$ we get
\[
d_{2+\delta}(\phi_{j+1}^N, g_\infty) \leq d_{2+\delta}(\phi_j^N,
 g_\infty)\, \frac {\frac r{\Delta t}}{\frac r{\Delta t}
-(2+\delta)} \left( \Delta t (p^{2+\delta} + q^{2+\delta} -1) +
1\right )
\]
and remembering that
\[
S_{p,q}(\delta)=(p^{2+\delta} + q^{2+\delta} -1) + \frac
{2+\delta}2 (1-p^2-q^2)
\]
we get
\[
\begin{aligned}
& \frac {\frac r{\Delta t}}{\frac r{\Delta t} - (2+\delta)}
\left( \Delta t (p^{2+\delta} + q^{2+\delta} -1) + 1\right )\\
&= \frac{ \Delta t (p^{2+\delta} + q^{2+\delta} -1) + 1}
{1- \Delta t \,\frac{2+\delta} 2\, (1-p^2-q^2)}\\
& =\frac{ 1+ S_{p,q}(\delta)\, \Delta t + (\Delta t)^2 \frac
{2+\delta}2 (1-p^2-q^2)(p^{2+\delta} + q^{2+\delta} -1)  } {1-
(\Delta t)^2 \left(\frac{2+\delta} 2 (1-p^2-q^2)\right )^2}.
\end{aligned}
\]
Since $S_{p,q}(\delta)<0$ and
\[
\frac 1{1-(\Delta t)^2 \left(\frac {\delta +2}2 (1-p^2-q^2)\right
)^2} = 1 +  o(\Delta t), \quad \Delta t \to 0,
\]
we get, for $N\in \N$ large enough and $j=1,\dots, N-1$:
\[
d_{2+\delta}(\phi_{j+1}^N, g_\infty) \leq \left(1 - \frac{|S_
{p,q}(\delta)|}2\Delta t\right ) d_{2+\delta}(\phi_j^N, g_\infty )
\leq \, d_{2+\delta}(\phi_j^N, g_\infty).
\]
Recursively, we get
\[
d_{2+\delta}(\phi_{j+1}^N, g_\infty) \leq d_{2+\delta}(g_0,
g_\infty)
\]
where $d_{2+\delta}( g_0,g_\infty) $ depends only on the
quantities in \eqref{momenti}.  Therefore, for $j=1,\dots, N$ and
$\xi\neq 0$ we get
\[
|\hat \phi_j^N (\xi)| \leq |\hat g_\infty(\xi)| +
|\xi|^{2+\delta}d_{2+\delta} (\phi_{j}^N, g_\infty ) \leq |\hat
g_\infty(\xi)| +  |\xi|^{2+\delta}  C
\]
for $C=d_{2+\delta}( g_0,g_\infty) $ independent of $N$ and $j$.
Since
\[
\hat g_\infty(\xi) = 1- \frac {\xi^2} 2 + o(\xi^2), \quad \xi \to
0
\]
we get that for any $0<k<\frac 12$ there exists $\rho >0$ such
that for $j=0,\dots, N$
\[
|\hat \phi_j^N(\xi)| \leq \frac 1 {1+k \, \xi^2}, \quad |\xi| \leq
\rho
\]
which achieves the proof.
\end{proof}

We are now in position to prove Theorem \ref{teo-bounds}.

\renewcommand{\thetheonumbered}{\ref{teo-bounds}}
\begin{theonumbered}
 Assume $0<q\leq p$  satisfying $p^2+q^2<1$ and such that there is $\tilde\delta \in (0,1)$ for which  $\SSS(\delta) <0$, for $0<\delta<\tilde\delta$.
 Let $g(t)$ be the weak
solution of the equation (\ref{g-eq}), corresponding to the
initial density $g_0$ satisfying the normalization conditions
\fer{norm1}, and
$$
\int_{\R}|v|^{2+\tilde\delta}\, g_0(v)\, \d v<+\infty.
$$
If in addition
 \begin{equation*}\tag{\ref{con3}}
|\hat g_0(\xi)|\leq \frac 1{(1+\beta|\xi|)^\nu}, \quad  |\xi|>R,
\end{equation*}
for some $R>0$, $\nu >0$ and $\beta>0$, then there exist
$\rho>0$,\ $k >0$,\ $\beta' >0$, $\nu'>0$
 such that
$g(t)$ satisfies
 \begin{equation*}
\tag{\ref{bbs}}
 |\hat g(\xi, t)| \leq
\left\{
\begin{aligned}
 &\frac 1{1+ k\xi^2},\quad |\xi| \leq \rho ,\quad t\geq 0\\
& \frac 1{(1+ \beta ' |\xi|)^{{\nu}'}},\quad |\xi|> \rho,\quad
t\geq 0.
\end{aligned}
\right .
\end{equation*}
\end{theonumbered}

\begin{proof}
 The bound on
the low frequences $|\xi|\leq \rho$ has been established in Lemma
\ref{close}. Moreover, as a consequence of Proposition 2.4 in
\cite{DFT}, we can suppose that condition (\ref{con3}) holds for
any $|\xi| >\rho$ with a possibly smaller exponent $\nu'$. We will
prove that for any $N\in \N$ and $j=0,\dots, N$ we get
\[
|\hat \varphi^N_j(\xi)| \leq \frac 1{(1+ \beta'|\xi|)^{\nu'}},
\quad  |\xi|>\rho
\]
for  positive $\nu'$ and $\beta'$ small enough. By induction we
have only to check the bound on
\begin{equation*}
\hat \varphi^N_1(\xi)=\frac {r}{\Delta
t}\int_1^{+\infty}\left[\Delta t\ \hat g_0(p\tau\xi)\hat
g_0(q\tau\xi)+(1-\Delta t)\ \hat g_0(\tau\xi)\right]\frac
1{\tau^{\frac r{\Delta t}+1}}\ \d \tau
\end{equation*}
Let $|\xi|> \rho$ and $\tau>1$. We are faced with three different
situations.
\begin{enumerate}[{\rm Case} I:]
\item If $|q\tau\xi|>\rho$, then
\begin{equation*}
\begin{aligned}
|\hat g_0(p\tau\xi)\hat g_0(q\tau\xi)|&\leq
\frac{1}{(1+{\beta}p\tau |\xi|)^{\nu'}}
\frac{1}{(1+{\beta}q\tau|\xi|)^{\nu'}}\\
&\leq\frac{1}{(1+{\beta}(p+q) |\xi|)^{\nu'}}\\
&\leq\frac{1}{(1+{\beta}' |\xi|)^{\nu'}}\\ \ \ \ \
\end{aligned}
\end{equation*}
for $\beta'\leq \beta(p+q)$.

\item  If $|p\tau\xi|> \rho$ and $|q\tau\xi|\leq \rho$, then
\begin{equation*}
\begin{aligned}
|\hat g_0(p\tau\xi)\hat  g_0(q\tau\xi)|&\leq
\frac{1}{(1+{\beta}p\tau |\xi|)^{\nu'}}
\frac{1}{{(1+kq^2\tau^2|\xi|^2)}}\\
&\leq\frac{1}{(1+{\beta}p
 |\xi|)^{\nu'}}
\frac{1}{(1+kq^2|\xi|^2)}.\ \ \ \\
\end{aligned}
\end{equation*}
By choosing $\nu'>0$ small enough we can show that
\begin{equation*}
\frac{1}{(1+{\beta}p
 |\xi|)^{\nu'}}
\frac{1}{(1+kq^2|\xi|^2)}\leq \frac{1}{(1+{\beta}
 |\xi|)^{\nu'}}.
\end{equation*}
Indeed, since  $\frac{1+\beta x}{1+\beta px}\leq \frac 1p$ for any
$x\geq 0$ we obtain
\begin{equation*}
\begin{aligned}
\left(\frac{1+\beta |\xi|}{1+\beta p|\xi|}\right)^{\nu'}
\frac{1}{(1+kq^2|\xi|^2)}&\leq \left(\frac
1p\right)^{\nu'}\frac{1}{(1+k
 q^2|\xi|^2)}\\ &\leq \left(\frac 1p\right)^{\nu'}\frac{1}{(1+k
 q^2\rho^2)}.
 \end{aligned}
\end{equation*}
Finally the last term is smaller than 1 for  $\nu'\leq \log_{\frac
1p}(1+kq^2\rho^2)$.

\item  If $|p\tau\xi|\leq \rho$, then
\begin{equation*}
\begin{aligned}
|\hat  g_0(p\tau\xi)\hat g_0(q\tau\xi)|&\leq
\frac{1}{(1+kp^2\tau^2|\xi|^2)}\frac{1}{(1+kq^2\tau^2|\xi|^2)}\\
&\leq\frac{1}{(1+kp^2|\xi|^2)}\frac{1}{(1+kq^2|\xi|^2)}\ \ \ \\
\end{aligned}
\end{equation*}
and we have to show that
$$
\frac{1}{(1+kp^2|\xi|^2)(1+kq^2|\xi|^2)}\leq
\frac{1}{(1+\beta|\xi|)^{\nu'}}.
$$
Since
$$
\frac{1}{(1+kp^2|\xi|^2)(1+kq^2|\xi|^2)}\leq
\frac{1}{(1+k(p^2+q^2)|\xi|^2)}
$$
and $p^2+q^2 =C >0$, in order to establish  the desired estimate
we have only to prove that for $\nu'$ small enough we have
$$
\frac{1}{1+kC|\xi|^2}\leq\frac{1}{(1+\beta|\xi|)^{\nu'}}
$$
or
\[
\nu'\log(1+\beta|\xi|)\leq \log(1+kC|\xi|^2)
\]
for any $|\xi|> \rho$. This is true since the function
$$
F(\xi)=\frac{\log(1+kC|\xi|^2)}{\log(1+\beta|\xi|)}
$$
 is bounded below by a positive constant on  $|\xi|> \rho$
and so by choosing $\nu'>0$ smaller than this constant  we get the
desired estimate.
\end{enumerate}

 \end{proof}
\begin{rem}\label{remark}
 The bounds in \eqref{bbs} are also equivalent to
\[
 |\hat g(\xi, t)| \leq
\frac C{(1+ \kappa |\xi|)^{\mu}},\quad \xi\in \R, \quad t\geq 0
\]
for some positive constants $\kappa$, $\mu$ and $C$ suitably
chosen.
\end{rem}
\begin{rem}
 It wouldn't have been  possible to prove also the behavior of the solution $g(t)$ on the low frequencies by induction and this is why we had to
exploit the $d_{2+\delta}$ convergence and the behavior of the
steady state as in Lemma \ref{close}.
\end{rem}

\section{Smoothness of the steady state}\label{Smoothness}
 In this section we prove the following Theorem \ref{thm.smoothness}.

\renewcommand{\thetheonumbered}{\ref{thm.smoothness}}
\begin{theonumbered}
 Assume $0<q\leq p$ satisfying $p^2+q^2 <1$ and such that there is $\tilde\delta \in (0,1)$ for which  $\SSS(\delta) <0$
 for $0<\delta <\tilde \delta$, so that a non-trivial steady state $g_\infty$ to the Boltzmann equation \eqref{g-eq} exists.
Let us denote $\lambda \in (0,2)$ the exponent such that
$p^\lambda +q^\lambda =1$.
  Then $g_\infty$ is a smooth function and belongs to the $\lambda$-th Gevrey class $G^\lambda(\R)$, i.e.
  \begin{align*}
    \big| \hat g_\infty(\xi)\big| &\leq \exp\big(-\mu|\xi|^\lambda\big), \quad |\xi|>\rho
  \end{align*}
with suitable positive numbers $\rho$ and $\mu$.
\end{theonumbered}
\begin{proof}
We are going to prove that $g_\infty \in G^\lambda(\R)$, following
a scheme already used in \cite{MT}.

As in the proof of Lemma \ref{close}, we know the behavior of
$\hat g_\infty$ close to zero. In particular, by the same
reasoning, there exist $k\in (0,\frac12)$ and $\rho >0$ such that
\[
| \hat g_\infty (\xi)| \leq e^{-k\xi^2},\quad |\xi| \leq \rho.
\]
As a consequence of Proposition 2.4 in \cite{DFT} and without any
loss in generality, we can also suppose $\rho>1$. Let us remember
that $\hat g_\infty$ satisfies the equation
\[
 \frac 12(p^2+q^2-1) \,\xi\, \frac{d\hat g_\infty(\xi)}{d\xi} = \hat g_\infty(p\xi) \hat g_\infty(q\xi) -\hat g_\infty (\xi)
\]
and so, denoting again ${r =\frac 2{1-p^2-q^2}}$ and following
\cite{BC03}, we get
\[
\hat g_\infty(\xi) = r \int_{1}^\infty \frac {\hat g_\infty(p\tau
\xi) \hat g_\infty(q\tau \xi)}{\tau^{r+1}}\, \d\, \tau, \quad
\xi\in\R.
\]
Let us consider the space
\[
 X_\rho= \left\{ \psi \in L^\infty(\R), \ |\psi(\xi)|\leq 1, \ \psi(\xi) = \hat g_\infty (\xi) \text{\ for\ } |\xi|\leq \rho\right \}
\]
endowed with the metric $d_{2+\delta}$ defined in \eqref{ds},
where $\delta$ satisfies the assumptions of Theorem
\ref{thm.smoothness}. The space $X_\rho$ is a Fréchet space. Let
us consider then the function $R$ defined by \be
 R(\psi)(\xi)=
\begin{cases}
 \hat g_\infty(\xi) & |\xi| \leq \rho\\
r \int_{1}^\infty \frac {\psi(p\tau \xi) \psi(q\tau
\xi)}{\tau^{r+1}}\, \d\, \tau, & |\xi| >\rho.
\end{cases}
\ee We are going to prove that $R$ is a contraction on $X_\rho$
and so $\hat g_\infty$ is its unique fixed point. It is
straightforward that $R :X_\rho \longrightarrow X_\rho$. As for
the contractiveness, for $\psi$ and $\phi \in X_\rho$ and $|\xi|
>\rho$ we have
\[
\begin{aligned}
\frac { |R\left(\psi\right )(\xi) -R\left(\phi\right
)(\xi)|}{|\xi|^{2+\delta}}& \leq \left(r \int_{1}^\infty \frac
{\tau^{2+\delta} (p^{2+\delta} +q^{2+\delta})}{\tau^{r+1}}\, \d \,
\tau\right ) d_{2+\delta}(\psi, \phi)
\\
& \leq \left(r\, (p^{2+\delta} +q^{2+\delta})\int_{1}^\infty \frac
{\d \, \tau }{\tau^{r-1-\delta}}\, \right ) d_{2+\delta}(\psi,
\phi).
\end{aligned}
\]
We remark that the assumption
\[
 S_{p,q}(\delta) = p^{2+\delta} +q^{2+\delta} -1 + \frac {2+\delta}2(1-p^2-q^2) <0
\]
implies
\[
 \frac {2+\delta}2(1-p^2-q^2) <1 -(p^{2+\delta} +q^{2+\delta}) <1
\]
and so,
\[
  \frac {\delta}2(1-p^2-q^2) < p^2+q^2
\]
which is precisely
\[
 r-1-\delta >1.
\]
So we get
\[
 d_{2+\delta}(R(\psi), R(\phi)) \leq \frac{r\, (p^{2+\delta} +q^{2+\delta})}{r -2-\delta} d_{2+\delta}(\psi, \phi).
\]
Remembering  that $r = \frac 2{1-p^2-q^2}$, we get
\[
 \frac{r\, (p^{2+\delta} +q^{2+\delta})}{r -2-\delta} = \frac 2{1-p^2-q^2} (p^{2+\delta} +q^{2+\delta})\,
\frac 1{ \frac 2{1-p^2-q^2}-2-\delta} <1 \iff S_{p,q}(\delta) <0
\]
and this allows to conclude. Choosing $\psi_0\in X_\rho$ for
example as
\[
 \psi_0(\xi) = \begin{cases}
                 \hat g_\infty (\xi) & |\xi| \leq \rho\\
                     0 & |\xi| >\rho
                \end{cases}
\]
and defining by induction for $n\geq 0$
\[
\psi_{n+1} (\xi) =
\begin{cases}
 \hat g_\infty (\xi) & |\xi| \leq \rho\\
r \int_{\tau=1}^\infty \frac {\psi_n(p\tau \xi) \psi_n(q\tau
\xi)}{\tau^{r+1}}\, \d\, \tau, & |\xi| >\rho,
\end{cases}
\]
we get authomatically  $d_{2+\delta}(\psi_n, \hat g_\infty) \to 0$
for $n\to +\infty$, which implies convergence pointwise. We show
now that there exists $\mu >0$ such that for all $n\in\N$
\[
 |\psi_n(\xi)| \leq e^{-\mu|\xi|^\lambda}, \quad |\xi| >\rho
\]
 and this uniform estimate passes therefore to the limit and allows to conclude.
The only thing to control is that for $\tau >1$ and $|\xi|>\rho$
we have
\[
 |\psi_0(p\tau \xi) \psi_0(q\tau\xi)| \leq e^{-\mu |\xi|^\lambda}.
\]
We distinguish three cases:
\begin{enumerate}[{\rm Case} I:]
  \item If  $q\tau|\xi|>\rho$, since $p^\lambda +q^\lambda =1$ and  $\rho>1$ we get
    \begin{align*}
      e^{\mu|\xi|^\lambda}\cdot|\psi_0(p\tau\xi)|\cdot|\psi_0(q\tau\xi)| & \leq e^{\mu|\xi|^\lambda(1- \tau^\lambda (p^\lambda +q^\lambda))} \leq
e^{\mu|\xi|^\lambda(1-\tau^\lambda)}\leq 1.
    \end{align*}
  \item If $p\tau|\xi|\leq \rho$, then $|\xi| >1$ implies  $\xi^2 \geq |\xi|^\lambda$. Denoting  $p^2+q^2 =C$,  we conclude
    \begin{align*}
      e^{\mu|\xi|^\lambda}\cdot|\psi_0(p\tau\xi)|\cdot|\psi_0(q\tau\xi)| &
    \leq e^{\mu|\xi|^\lambda-k \tau^2 \xi^2 (p^2+ q^2)}
    \leq  e^{\mu|\xi|^\lambda-k  |\xi|^\lambda (p^2+ q^2)}
\leq e^{|\xi|^\lambda(\mu-k\, C )} \leq 1,
    \end{align*}
    provided that $\mu\leq k\, C$.
  \item Now assume that $q\tau|\xi|\leq\rho$ while $p\tau|\xi|>\rho$.
    Using the condition $p^\lambda +q^\lambda =1$ once again, one finds
    \begin{align*}
      e^{\mu|\xi|^\lambda}\cdot|\psi_0(q\tau\xi)|\cdot|\psi_0(p\tau\xi)| &
      \leq e^{\mu|\xi|^\lambda-k\, \tau^2 q^2\xi^2-\mu \tau^\lambda p^\lambda|\xi|^\lambda} \\
&\leq e^{|\xi|^\lambda( \mu(1-p^\lambda)- k q^2)} \leq
e^{|\xi|^\lambda( \mu q^\lambda- k q^2)} \leq 1
\end{align*}
provided that $\mu\leq \frac {k q^2}{q^\lambda}$.
    \end{enumerate}

\end{proof}

\begin{rem}
By the proof of Theorem \ref{thm.smoothness}, the condition
$g_\infty \in G^\lambda (\R)$ seems to be sharp, since in the
estimate of Case I, it wouldn't have been possible to replace
$\lambda$ by $\sigma >\lambda$. Moreover, in the case $p+q=1$ the
explicit stationary state is  $\hat g_\infty(\xi)= (1+|\xi|)
e^{-|\xi|}$ and this proves sharpness at least in this special
case.
It is worth noticing that Bobylev and Cercignani in \cite{BC03}
(Theorem 5.3) proved that for $p+q>1$ and $p^2+q^2 <1$ the
stationary state $g_\infty$ satisfies the bounds
\[
 e^{-\frac{\xi^2}2} \leq |\hat g_\infty(\xi)| \leq (1+|\xi|) e^{-|\xi|}.
\]
For the values of $p$ and $q$ for which in addition there is
$\tilde \delta \in (0,1)$ such that $S_{p,q}(\delta) <0$ for
$0<\delta<\tilde\delta$, our result improves the upper bound and
gives also a new result for some $p$ and $q$ in the range $p+q<1$.
\end{rem}
By a careful reading of both proofs of Theorem \ref{teo-bounds}
and Theorem \ref{thm.smoothness}, we can deduce that not only
polynomial tails of the Fourier transform of the initial data
$g_0$ are uniformly propagated by the solution $g(t)$, but also
exponential ones, as long as the exponent does not exceed the
exponent $\lambda$ characterizing the mixing parameters $p$ and
$q$. More precisely, the result is as follows.

\renewcommand{\thetheonumbered}{\ref{gev-prop}}
\begin{theonumbered}
 Assume $0<q\leq p$ satisfying $p^2+q^2 <1$ and such that there is $\tilde \delta \in (0,1)$ for which  $\SSS(\delta) <0$ for $0<\delta <\tilde \delta$.
Let us denote $\lambda \in (0,2)$ the exponent such that
$p^\lambda +q^\lambda =1$. Let $g(t)$ be the weak solution of the
equation (\ref{g-eq}), corresponding to the initial density $g_0$
satisfying the normalization conditions \fer{norm1}, and
$$
\int_{\R}|v|^{2+\tilde\delta}\, g_0(v)\, \d v<+\infty.
$$
If in addition
 \begin{equation*}
|\hat g_0(\xi)|\leq e^{-\beta |\xi|^\nu}, \quad  |\xi|>R,
\end{equation*}
for some $R>0$, $\nu >0$ and $\beta>0$, then there exist $\rho>0$
and $\kappa >0$,
 such that
$g(t)$ satisfies
 \begin{equation*}
 |\hat g(\xi, t)| \leq
\left\{
\begin{aligned}
& e^{-\kappa \xi^2},\quad |\xi| \leq \rho,\quad t\geq 0\\
& e^{-\kappa |\xi|^{\min (\nu, \lambda)}} ,\quad |\xi|> \rho,\quad
t\geq 0.
\end{aligned}
\right .
\end{equation*}
\end{theonumbered}

\begin{proof}
The proof follows the same lines as the proof of Theorem
\ref{teo-bounds}, replacing the polynomial decreasing by the
exponential one. The key argument for the low frequencies is that
the normalization assumptions \eqref{norm1} on the steady state
imply
\[
\hat g_{\infty} (\xi)  = 1-\frac {\xi^2}2 + o(\xi^2),\quad \xi\to
0,
\]
which means
\[
|\hat g_\infty(\xi)| \leq \frac 1{1+ k \xi^2},\quad |\xi| \leq
\rho
\]
or equivalently
\[
 |\hat g_\infty(\xi)| \leq e^{- k \xi^2},\quad |\xi| \leq \rho
\]
for any $k\in (0,\frac 12)$ and $\rho$ depending on $k$. The whole
proof follows then without any particular difficulty, exploiting
for the high frequencies the estimates performed in the proof of
the Gevrey regularity  of the steady state.
\end{proof}

\section{Strong convergence}\label{sec-strong}

In this section, we are going to prove Theorem \ref{teo-L1} on the
strong $L^1$ convergence of the scaled solution $g(t)$ to the
stationary state $g_\infty$.

\renewcommand{\thetheonumbered}{\ref{teo-L1}}
\begin{theonumbered}
Assume $0<q\leq p$ satisfying $p^2+q^2 <1$ and such that there
exists  $\tilde \delta \in (0,1)$ for which $\SSS (\delta) <0$ for
$0< \delta < \tilde \delta$
 and let  $g_\infty$ be the unique stationary solution of
\eqref{g-eq}.
 Let the initial
density $g_0$ satisfy  the normalization conditions \fer{norm1},
and
$$
\int_{\R}|v|^{2+\tilde\delta}\, g_0(v)\, \d v<+\infty.
$$
If in addition $g_0 \in {H^\eta(\R)}$ for some $\eta > 0$,
$\sqrt{g_0} \in {\dot H^{\nu}(\R)}$ for some $\nu > 0$, then the
solution $g(t)$ of (\ref{g-eq}) converges strongly in $L^1$ with
an exponential rate towards the stationary solution $g_\infty$,
i.e., there exist positive constants $C$ and $\gamma$ explicitly
computable such that
 \[
 \Vert g(t) - g_\infty \Vert_{L^1(\R)} \le C e^{-\gamma t},\quad t\geq 0.
 \]
\end{theonumbered}
Let us begin by the following lemma.

\begin{lem}\label{fisher}
Let the initial density $g_0$ satisfy the normalization conditions
\fer{norm1}, and
$$
\int_{\R}|v|^{2+\tilde\delta}\, g_0(v)\, \d v<+\infty.
$$
If in addition $\sqrt{g_0} \in {\dot H^{\nu}(\R)}$ for some $\nu>
0$, then $g_0$ satisfies
\[
|\hat g_0(\xi)|\leq \frac C{(1+\beta|\xi|)^\nu}, \quad  \xi\in\R
\]
for positive constants $C$ and $\beta$ and the solution $g(t)$ of
(\ref{g-eq}) satisfies \be\label{converg} \sup_{\xi\in\R} |\hat
g(\xi, t) -\hat g_\infty(\xi)| \leq C_1 e^{-C_2 t},\quad t\geq 0
\ee for positive constants $C_1$ and $C_2$.
\end{lem}

\begin{proof}
Since $g_0= \sqrt{g_0} \sqrt{g_0}$, then $\hat g_0 = \widehat
{\sqrt{g_0}} \ast  \widehat {\sqrt{g_0}}$. So, for $\xi\in\R$ we
get
\[
\begin{aligned}
 |\xi|^{\nu} |\hat g_0(\xi)| &\leq \int_\R |\xi|^\nu \left|\widehat {\sqrt{g_0}}(\xi -\tau)\widehat {\sqrt{g_0}}(\tau)\right |\, \d \tau\\
&\leq K \int_\R \left(|\xi-\tau|^\nu +|\tau|^\nu\right ) \left|\widehat {\sqrt{g_0}}(\xi -\tau)\widehat {\sqrt{g_0}}(\tau)\right |\, \d \tau\\
&\leq  2K \| \sqrt{g_0}\|_{\dot H^\nu}.
\end{aligned}
\]
Since moreover $\left|\hat g_0(\xi)\right |\leq 1$, we can find
positive  $C$ and $\beta$ such that
\[
|\hat g_0(\xi)|\leq \frac C{(1+\beta|\xi|)^\nu}, \quad \xi\in\R.
\]
Thanks to Theorem \ref{teo-bounds} and to  Remark \ref{remark}, we
get
\[
 |\hat g(\xi, t)| \leq
\frac {\tilde C}{(1+ \kappa |\xi|)^{\mu}},\quad \xi\in \R, \quad
t\geq 0
\]
for suitable $\tilde C$, $\kappa$ and $\mu$. The steady state
$g_\infty$ belongs to a Gevrey class, so it satisfies an analogous
estimate, with suitable constants which we can suppose to be the
same. Let now $R>0$ to be chosen in a moment. We get, for
$\xi\in\R$:
\[
 |\hat g(\xi,t) - \hat g_\infty(\xi)| \le d_{2+\delta}(g(t), g_\infty) R^{2+\delta} +
 \frac {2\tilde C}{(\kappa R)^\mu},
 \]
which implies, optimizing over $R$,
\[
|\hat g(\xi,t) - \hat g_\infty(\xi)| \le C_1 d_{2+\delta}(g(t),
g_\infty)^{\mu/(2+\delta+\mu)} = C_1 e^{-C_2 t},\quad \xi\in\R,\
t\geq 0.
\]
for $C_1$ and $C_2$ positive constants.
\end{proof}

\noindent  {\bf Proof of Theorem \ref{teo-L1}}.
 The result consists in converting the weak convergence
in the Fourier distance  of the solution $g(t)$ to the stationary
state $g_\infty$ (Theorem \ref{teo-pt}) into a $L^1$ convergence
by interpolating this weak distance $d_{2+\delta}$ with the
uniformly boundedness in time of suitable moment and Sobolev norm
of the solution itself. The only missing ingredient at this point
is the boundedness of the Sobolev norm and we will go through the
proof of it in a moment. Let us recall first how we can
interpolate these results, following the scheme introduced in
\cite{CGT} and fruitfully applied afterward in several papers
(\cite{BCT}, \cite{CT} for instance).

First of all, it is easy to prove the  two following
interpolation bounds (see Theorems 4.1 and 4.2 in \cite{CGT}): for
$\delta \in (0,\tilde \delta)$,  there exists a  positive constant
$C$ such that
\[
\|h\|_{L^1} \leq C \| |v|^{2+\delta} h\|^{\frac
1{1+2(2+\delta)}}_{L^1} \|h\|^{\frac
{2(2+\delta)}{1+2(2+\delta)}}_{L^2}
\]
and for any $s\geq 0$ there exist positive constants $M$, $N$,
$\beta$ and $\gamma$ such that \be\label{sob-inter} \|h\|_{H^s}
\leq C \left(\sup_\R \frac{|\hat
h(\xi)|}{|\xi|^{2+\delta}}\right)^{\beta} \left( \|h\|_{H^M} +
\|h\|_{H^N}\right )^{\gamma}. \ee So, letting $h= g(t)-g_\infty$,
and $s=0$ we get
\begin{multline*}
\|g(t)-g_\infty\|_{L^1} \leq  C
\left( \| |v|^{2+\delta}g(t)\|_{L^1}  +\| |v|^{2+\delta} g_\infty\|_{L^1}\right)^{\tilde \alpha} \times \\
 d_{2+\delta}(g(t), g_\infty)^{\tilde \beta} \, \left( \|g(t)\|_{H^M} + \|g(t)\|_{H^N} +
  \|g_\infty\|_{H^M} + \|g_\infty\|_{H^N}\right )^{\tilde \gamma}
\end{multline*}
for suitable exponents $\tilde \alpha$, $\tilde \beta$, $\tilde
\gamma$. Concerning the stationary state $g_\infty$,  it have been
proved in Theorem \ref{teo-pt} that $\| |v|^{2+\delta}
g_\infty\|_{L^1} <\infty$ for $\delta \in (0,\tilde \delta)$ and
thanks to the Gevrey regularity, we have $g_\infty \in H^s(\R)$
for all $s\geq 0$. As for the scaled solution $g(t)$, the uniform
boundedness of the $(2+\delta)$-th moment  have been also proved
in Theorem \ref{teo-pt}, so we will get the $L^1$ exponential
convergence as soon as we prove the uniform boundedness in time of
$g(t)$ in a suitable Sobolev space $H^{\max (M,N)}(\R)$. Of
course, we need to assume  $g_0\in H^{\max(M,N)}(\R)$. As a
byproduct of the uniform boundedness of Sobolev norms, we will
also get from \eqref{sob-inter} the convergence of $g(t)$ to
$g_\infty$  in Sobolev spaces.

Let us recall how to prove the uniform boundedness of $g(t)$ in a
generic homogenous Sobolev space $\dot H^\eta$ for $\eta \geq 0$
under the assumption $g_0\in \dot H^\eta$. First of all, let us
remark that $g(t) \in \dot H^\eta$ for all $t >0$, without any
uniformity in time. Indeed, coming back to the original non scaled
solution $f(t)$ we get
\[
\begin{aligned}
\frac{\d}{\d t} \| f(t)\|^2_{\dot H^\eta} &= \frac{\d}{\d t}\int_\R |\xi|^{2\eta} \hat f(\xi, t) \overline {\hat f(\xi,t)} \, \d \xi \\
&=
2\int_\R |\xi|^{2\eta} \mathrm{Re} \left(\overline {\hat f(\xi, t)} \partial_t \hat f(\xi, t)\right ) \, \d \xi  \\
& = 2\int_{\R} |\xi|^{2\eta} \mathrm {Re}\left( \overline {\hat
f(\xi, t)}
\left( \hat f(p\xi, t) \hat f(q \xi, t) -\hat f(\xi, t) \right )\right ) \, \d \xi\\
&=-2\int_\R  |\xi|^{2\eta} |\hat f(\xi, t)|^2\, \d \xi + 2 \int_\R
|\xi|^{2\eta} \mathrm{Re}\left( \overline {\hat f(\xi, t)} \hat
f(p\xi, t) \hat f(q \xi, t)\right )
\, \d \xi \\
& \leq -1\int_\R  |\xi|^{2\eta} |\hat f(\xi, t)|^2\, \d \xi +
\int_\R |\xi|^{2\eta}  |\hat f(p\xi,t)|^2 |\hat f(q\xi,t)|^2\, \d \xi\\
&\leq  -1\int_\R  |\xi|^{2\eta} |\hat f(\xi, t)|^2\, \d \xi +
\frac 12 \left(\frac 1{q^{2\eta +1}} + \frac  1{p^{2\eta
+1}}\right ) \int_\R |\xi|^{2\eta}  |\hat f(\xi,t)|^2\, \d \xi .
\end{aligned}
\]
Since $\hat g\left(e^{\frac{p^2+q^2-1} 2t}\xi, t\right )=  \hat f
\left (\xi, t\right )$ we get
\[
\frac{\d}{\d t} \left( e^{\frac{1-p^2-q^2} 2\,(2\eta +1)\, t} \|
g(t)\|^2_{\dot H^\eta} \right )\leq \left(-1 +\frac 12 \left(\frac
1{q^{2\eta +1}} + \frac  1{p^{2\eta +1}}\right )\right )
e^{\frac{1-p^2-q^2} 2\,(2\eta +1)\, t} \| g(t)\|^2_{\dot H^\eta}
\]
and so \be\label{stima-imprecisa} \frac{\d}{\d t} \|
g(t)\|^2_{\dot H^\eta} \leq \left[-1- \frac {1-p^2-q^2} 2 (2\eta
+1) + \frac 12 \left(\frac 1{q^{2\eta +1}} + \frac  1{p^{2\eta
+1}}\right )\right ]\| g(t)\|^2_{\dot H^\eta}  = C_{p,q,\eta} \|
g(t)\|^2_{\dot H^\eta}, \ee which leads to
\[
\| g(t)\|^2_{\dot H^\eta} \leq \| g_0\|^2_{\dot H^\eta}
e^{C_{p,q,\eta} t}.
\]
Since $p^2+q^2 <1$, it is not difficult  to be convinced that
$C_{p,q,\eta} >0$ as soon as for example $q$ is small enough.

Let us make  estimate \eqref{stima-imprecisa} more accurate. The
goal is to get for example the following differential inequality:
for two positive constants $H$ and $K$ and $t_0>0$:
\be\label{unif} \frac{\d}{\d t} \| g(t)\|^2_{\dot H^\eta} \leq - H
\|g(t)\|^2_{\dot H^\eta} + K,\quad  t\geq t_0 \ee so that
\[
\| g(t)\|^2_{\dot H^\eta} \leq C \max (\|g(t_0)\|^2_{\dot H^\eta},
1), \quad t\geq t_0.
\]
Let us come back to inequality
\[
\frac{\d}{\d t} \| f(t)\|^2_{\dot H^\eta} \leq -1\int_\R
|\xi|^{2\eta} |\hat f(\xi, t)|^2\, \d \xi + \int_\R |\xi|^{2\eta}
|\hat f(p\xi,t)|^2 |\hat f(q\xi,t)|^2\, \d \xi
\]
which reads, on the scaled solution $g(t)$,
\[
\begin{aligned}
&\frac{\d}{\d t} \left( e^{\frac{1-p^2-q^2} 2\,(2\eta +1)\, t}
 \| g(t)\|^2_{\dot H^\eta}\right )\\
& \leq -e^{\frac{1-p^2-q^2} 2\,(2\eta +1)\, t}
 \int_\R  |\xi|^{2\eta} |\hat g(\xi, t)|^2\, \d \xi + e^{\frac{1-p^2-q^2} 2\,(2\eta +1)\, t}
\int_\R |\xi|^{2\eta}  |\hat g(p\xi,t)|^2 |\hat g(q\xi,t)|^2\, \d
\xi
\end{aligned}
\]
and so
\[
\frac{\d}{\d t} \| g(t)\|^2_{\dot H^\eta} \leq \left(-1- \frac
{1-p^2-q^2} 2 (2\eta +1)\right ) \| g(t)\|^2_{\dot H^\eta} +
\int_\R |\xi|^{2\eta}  |\hat g(p\xi,t)|^2 |\hat g(q\xi,t)|^2\, \d
\xi.
\]
Since $p^2+q^2<1$,  it would be enough to obtain for example the
following inequality \be \label{dis1} \int_\R |\xi|^{2\eta} |\hat
g(p\xi, t)|^2 |\hat g(q \xi, t)|^2\, \d \xi \leq \frac 12 \|
g(t)\|_{\dot H^{\eta}}^2 + K, \quad t\geq t_0 \ee where  $K>0$ is
independent of $t$. Let us prove inequality \eqref{dis1}. We split
the integral in \eqref{dis1} into two parts
\[
\int_\R |\xi|^{2\eta} |\hat g(p\xi, t)|^2 |\hat g(q \xi, t)|^2\,
\d \xi = \int_{|\xi| \leq R} + \int_{|\xi|>R} = A+B
\]
where $R$ will be chosen later. Let us estimate first the  term in
$A$ (we will denote $\eps$ a constant which is allowed to vary
from one line to another, depending at most on $p$ and $q$). Since
$|\hat g(\xi, t)| \leq 1$ for $\xi\in\R$ and $t\geq 0$ we simply
get
\[
\begin{aligned}
&\int_{|\xi| \leq R} |\xi|^{2\eta} |\hat g(p\xi, t)|^2 |\hat g(q \xi, t)|^2\, \d \xi\\
&\leq \int_{|\xi| \leq R} |\xi|^{2\eta}\, \d \xi = \frac 2{2\eta
+1} R^{2\eta +1}, \quad t\geq 0.
\end{aligned}
\]
Let us come to the  term in $B$, where we are going to exploit
Lemma \ref{fisher}. We remark that $\hat g_\infty (\xi)\to 0$ for
$\xi \to +\infty$ and so, by Lemma \ref{fisher}, for any $\eps >0$
there exist $R>0$ and $t_0$ depending on $\eps$ and $p$ such that
\[
 |\hat g(p\xi, t)| \leq |\hat g(p\xi,t)-\hat g_\infty(p\xi)| + |\hat g_\infty(p\xi)| \leq 2\eps,  \quad |\xi| >R, \ t\geq t_0.
\]
We can deduce  for $t\geq t_0$:
\[
\begin{aligned}
 &\int_{|\xi| >R} |\xi|^{2\eta} |\hat g(p\xi, t)|^2 |\hat g(q \xi, t)|^2\, \d \xi
\leq (2\eps)^2 \int_{|\xi| >R} |\xi|^{2\eta} |\hat g(q \xi, t)|^2\, \d \xi\\
& \leq \frac {\eps}{q^{2\eta +1}} \int_{\R} |\xi|^{2\eta} |\hat
g(\xi,t)|^2 \, \d \xi =\eps \| g(t)\|_{\dot H^\eta}^2.
\end{aligned}
\]
We have obtained
\[
\begin{aligned}
\int_\R |\xi|^{2\eta} |\hat g(p\xi, t)|^2 |\hat g(q \xi, t)|^2\,
\d \xi &\leq \eps \| g(t)\|_{\dot H^{\eta}}^2 + \frac 2{2\eta
+1}R^{2\eta +1},  \quad t\geq t_0.
\end{aligned}
\]
Letting $\eps$ be fixed such that  $\eps \leq \frac 12$, we get
the desired estimate.

\hfill $\scriptstyle\square$

%


\section{Lyapunov functionals and open questions}\label{lyap}

The case $p+q =1$ separates in a natural way from the others. It
corresponds to a one-dimensional dissipative Boltzmann equation in
which the momentum is preserved in a microscopic collision of type
\fer{coll}. Equation \fer{eq:boltz} with $p+q=1$ as been
intensively studied in a series of papers \cite{BMP, BK, BC03,
PT06}. Among other properties, this model possesses an explicit
self similar solution, which has been first discovered in
\cite{BMP}. In fact, condition $p+q= 1$ implies
 \[
\frac 12\left(p^2+q^2-1\right) = -pq .
 \]
Hence the Fourier transformed version of the scaled equation
\fer{g-eq} can be written as
 \begin{equation}\label{g-eq1}
\hat g(p\xi,t) \hat g(q\xi,t)-\hat g(\xi,t) = \partial_t \hat
g(\xi,t) -pq \xi\partial_\xi\hat g(\xi,t).
 \end{equation}
 The choice
 \[
\hg_\infty(\xi) = (1 +|\xi|) e^{-|\xi| }
 \]
leads to
 \[
 \hat g_\infty(p\xi) \hat g_\infty(q\xi)-\hat g_\infty(\xi) = pq |\xi|^2e^{-|\xi|} = - pq \xi\partial_\xi\hat g_\infty(\xi)
 \]
 and so $\hg_\infty$ solves \fer{g-eq1} as a stationary solution for any choice of the parameters $p$
 and $q$ such that $p+q =1$. It can be easily verified that in physical
 variables the steady solution reads
  \be\label{ste2}
 g_\infty(v) = \frac 2{\pi(1+ v^2)^2} .
  \ee
 Note that this function satisfies the normalization conditions \fer{norm1}.
 Under these constraints, however, it can be shown \cite{To99} that $g_\infty$ is
 the (unique) minimizer of the convex functional
 \be
\label{Liap}
 H(f) = - \int_\R \sqrt{f(v)} \, \d v .
 \ee
It is a natural question to investigate whether the functional $H$
is a Lyapunov functional for the scaled equation for $g(v,t)$,
which can be formally written as
 \be\label{scal}
\partial_t g (v,t) = g_p* g_q(v, t) - g(v,t) + \frac 12\left(p^2+q^2-1\right)
\partial_v(vg(v, t)).
 \ee
 The results of both Section \ref{regularity} and Section \ref{sec-strong} lead
 to conclude that, under suitable regularity assumptions on the initial value,
 one can study the time derivative of the functional $H(g)(t)$ along solutions
 to equation \fer{scal}, obtaining
 \[
\frac {\d}{\d t} H(g)(t) = -\frac 12 \left\{ \int_\R
\frac{g_p*g_q(v,t)}{\sqrt {g(v,t)}}\, \d v - \frac{1+p^2
+q^2}2\int_\R {\sqrt {g(v,t)}}\, \d v \right\}.
 \]
Hence, the functional $H$ is a Lyapunov functional for equation
\fer{scal} provided the inequality
 \be \label{main-in}
\frac{1+p^2 +q^2}2\int_\R {\sqrt {g(v)}}\, \d v \le  \int_\R
\frac{g_p*g_q(v)}{\sqrt {g(v)}} \, \d v , \qquad p+q = 1,
 \ee
is verified for all functions satisfying constraints \fer{norm1}.
We remark that inequality \fer{main-in} is saturated by choosing
$g = g_\infty$, with $g_\infty$ defined as in \fer{ste2}. To our
knowledge, this inequality has never been investigated before, but
it can be conjectured that it holds true, even if we are not able
to prove it.

A different way to attach the problem is to resort to the Fourier
version of equation \fer{scal}. This idea has been fruitfully
employed in \cite{BT} to recover Lyapunov functionals for the
Boltzmann equation for Maxwell molecules. Let us consider the
approximate solution \fer{sol-d} which is a convex combination of
the probability densities $\hat g(\xi,t)$ and $\hat g(p\xi,t) \hat
g(q\xi,t)$. For any convex functional $\widehat H$ acting on $\hg$
we obtain
\begin{equation}\label{H-d}
\widehat H(\hat g(\xi, t+\Delta t)) \le \frac{r}{\Delta t}\
\int_{1}^{+\infty}\left(\Delta t \,\, \widehat H \left( \hat
g(\tau p\xi,t) \hat g(\tau q\xi,t)\right) \ +(1-\Delta t)\
\widehat H (\hat g( \tau \xi,t)) \right)\frac{\d \tau}{\tau^{\frac
r{\Delta t}+1}}.
\end{equation}
 If $\widehat H(\hg) = H(g)$ is defined by \fer{Liap}, we have
 \[
\widehat H(\hg(\tau \xi)) = \sqrt\tau  \widehat H(\hg(\xi)),
 \]
and this implies that inequality \fer{H-d} becomes
 \[
\widehat H(\hat g(\xi, t+\Delta t)) \le \frac{\frac{r}{\Delta
t}}{\frac{r}{\Delta t}- \frac 12} \left(\Delta t \,\, \widehat H
\left( \hat g( p\xi,t) \hat g(q\xi,t)\right) \ +(1-\Delta t)\
\widehat H (\hat g(  \xi,t)) \right).
\]
Let us suppose that there exists $A(p,q)$ such that for all
functions $g$ satisfying conditions \fer{norm1} we get
 \be \label{ine1}
\widehat H \left( \hat g( p\xi) \hat g(q\xi)\right) \le A(p,q)
\widehat H (\hat g(  \xi)).
 \ee
Then, since
 \[
\frac 1r=\frac{1-p^2-q^2}{2},
 \]
the condition
 \be\label{coeff}
A(p,q) = \frac{3 +p^2 +q^2}4
 \ee
would imply
 \[
\frac{\frac{r}{\Delta t}}{\frac{r}{\Delta t}- \frac 12}
\left(\Delta t \,\, A(p,q) \ +(1-\Delta t)\ \right) = 1.
 \]
 In this case, inequality \fer{H-d} would imply
 \[
\widehat H(\hat g(\xi, t+\Delta t)) \le \widehat H(\hat g(\xi,
t)),
 \]
and $H$ would be a Lyapunov functional. Note that inequality
\fer{ine1} corresponds to a reverse Young inequality first derived
by Leindler \cite{Lei}: for $0 < \alpha, \beta, \rho \le 1$ and
$f$, $g$ non-negative
 \cite{Ba}
 \be\label{Young}
 \Vert f*g\Vert_\rho \ge \Vert f\Vert_\alpha\Vert g \Vert_\beta, \qquad
 1/\alpha + 1/\beta = 1 + 1/\rho.
 \ee
In our case, $\rho = \alpha = 1/2$, $\beta =1$ together with the
second condition in \fer{norm1} implies
 \[
 \Vert f_p*f_q\Vert_{1/2} \ge p \Vert f\Vert_{1/2},
 \]
 namely inequality \fer{ine1} with $A(p,q) =  p$.
Unlikely, the direct
 application of inequality \fer{Young} is not enough to obtain \fer{ine1}. It
 remains an open question to prove that, under constraints \fer{norm1} it holds the
 Young-type reverse inequality
  \[
\Vert f_p*f_q\Vert^{1/2}_{1/2} \ge A(p,q) \Vert
f\Vert_{1/2}^{1/2},
 \]
 where $p+q =1$ and $A(p,q)$ is given by \fer{coeff}.


\section{Appendix}\label{annexe}

\noindent{\bf Proof of Proposition \ref{prop-momenti}:} We will
consider only the dissipative case $p^2+q^2<1$, since the other
case adapts straightforwardly. It is easy to get for all $N\in \N$
and $j=0,\dots, N-1$: \be \label{massa-mom} \phi_{j+1}^N(v) \geq
0, \quad \int_\R \phi_{j+1}^N(v)\, \d v=1\, \quad \int_\R v\,
\phi_{j+1}^N(v) \, \d v=0. \ee
 Let us consider therefore the evolution of the two other
moments of the sequence $\phi_j^N$. First of all, it is worth
noticing that for any function $h\in L^1(\R)$, $\alpha>0$ and
$\tau \neq 0$ we have \be\label{scaling} \int_\R |v|^{\alpha}
\frac 1 \tau h\left(\frac v\tau\right ) \d v= |\tau|^\alpha \|
|v|^\alpha h\|_{L^1}. \ee Let us compute now the second moment of
$\phi_{j+1}^N$:
\[
\begin{aligned}
\int_\R v^2 \phi_{j+1}^N(v)\, \d v&= \frac{r}{\Delta t}\
\int_{1}^{+\infty}\left[\Delta t \left(\int_\R v^2 \frac 1 \tau
\left(\phi_{j,p}^N\ast \phi_{j, q}^N\right ) \left(\frac v
\tau\right )\, \d v \right )+\right .\\
&\qquad\qquad\qquad +\left. (1-\Delta t) \left( \int_\R v^2  \frac
1 \tau
\phi_j^N\left(\frac v\tau\right )\, \d v \right ) \right]\frac{\d \tau}{\tau^{\frac r{\Delta t}+1}}\\
&= \frac{r}{\Delta t} \left(\Delta t \| v^2 \phi_{j,p}^N\ast
\phi_{j, q}^N\|_{L^1} +(1-\Delta t) \| v^2 \phi_j^N\|_{L^1} \right
) \int_{1}^{+\infty}\frac{ \tau ^2\d \tau}{\tau^{\frac r{\Delta
t}+1}}.
\end{aligned}
\]
For $N> \frac {2T}{r}$ we get
\[
\frac{r}{\Delta t}\int_{1}^{+\infty}\frac{ \d \tau}{\tau^{\frac
r{\Delta t}-1}} = \frac 1{ 1-\frac {2\Delta t}{r}}  = \frac 1{
1+\Delta t (p^2+q^2-1)}
\]
and so we are left with $\| v^2 \phi_{j,p}^N\ast \phi_{j,
q}^N\|_{L^1}$. We have
\[
\begin{aligned}
&\iint_{\R^2} v^2\, \frac 1p \phi_j^N\left( \frac {v-w}p\right )\, \frac 1q \phi_j^N\left ( \frac wq\right )\, \d w\, \d v=\\
&  \iint_{\R^2}\left((v-w)^2 + w^2 +2 (v-w)w\right ) \frac 1p
\phi_j^N\left( \frac {v-w}p\right ) \frac 1q \phi_j^N\left ( \frac
wq\right )\, \d w\, \d v.
\end{aligned}
\]
Thanks to \eqref{massa-mom}, we get
\[
\| v^2 \phi_{j,p}^N\ast \phi_{j, q}^N\|_{L^1} = (p^2+q^2) \| v^2
\phi_j^N\|_{L^1}
\]
and we end up with
\[
\| v^2 \phi_{j+1}^N\|_{L^1} = \frac{(\Delta t(p^2+q^2-1) +1)} {
1+\Delta t (p^2+q^2-1)}\| v^2 \phi_j^N\|_{L^1} = \| v^2
\phi_j^N\|_{L^1}
\]
and so by a recursive procedure \be \label{temp} \int_\R v^2
\phi_{j+1}^N(v) \, \d v=1, \quad j=0,\dots, N-1. \ee As for
$\int_\R |v|^{2+\delta}\, \phi_{j+1}^N(v)\, \d v$, we proceed in
the same way:
\[
\begin{aligned}
&\int_\R |v|^{2+\delta} \phi_{j+1}^N(v) \, \d v\\
&= \frac{r}{\Delta t}\ \int_{1}^{+\infty}\left[\Delta t
\left(\int_\R |v|^{2+\delta} \frac 1 \tau \left(\phi_{j,p}^N\ast
\phi_{j,
q}^N\right ) \left(\frac v \tau\right )\, \d v \right )+\right .\\
&\qquad\qquad\qquad +\left .(1-\Delta t) \left( \int_\R
|v|^{2+\delta}  \frac 1 \tau
\phi_j^N\left(\frac v\tau\right )\, \d v \right ) \right]\frac{\d \tau}{\tau^{\frac r{\Delta t}+1}}\\
&= \frac{r}{\Delta t} \left(\Delta t \| |v|^{2+\delta}\,
\phi_{j,p}^N\ast \phi_{j, q}^N\|_{L^1} +(1-\Delta t) \|
|v|^{2+\delta}\,  \phi_j^N\|_{L^1} \right )
\int_{1}^{+\infty}\frac{ \tau^{2+\delta}\d \tau}{\tau^{\frac
r{\Delta t}+1}}.
\end{aligned}
\]
Now, for $N> \frac {(2+\delta)T}{r}$ we get
\[
\frac{r}{\Delta t}\int_{1}^{+\infty}\frac{ \d \tau}{\tau^{\frac
r{\Delta t}-1-\delta}} = \frac 1{ 1-\frac {\Delta t}{r} (\delta
+2)}  = \frac 1{ 1-\Delta t \frac {\delta +2}2 (1-p^2-q^2)}.
\]
Let us estimate $\| |v|^{2+\delta} \phi_{j,p}^N\ast \phi_{j,
q}^N\|_{L^1}$. We will denote
\[
d_j^N = \int_\R |v|^{2+\delta} \phi_{j}^N(v) \, \d v.
\]
Since $0<\delta <1$, we can write
\[
\begin{aligned}
&|v|^{2+\delta} = v^2 |v|^{\delta} = (v-w+w)^2 |v-w+w|^{\delta}
\leq \left( (v-w)^2 + w^2 + 2(v-w)w\right )
\left( |v-w|^\delta +|w|^\delta\right )\\
& = |v-w|^{2+\delta} + w^2 |v-w|^\delta + 2 (v-w) |v-w|^\delta w +
(v-w)^2|w|^\delta + |w|^{2+\delta} + 2(v-w) |w|^\delta w
\end{aligned}
\]
so, thanks to \eqref{massa-mom} and \eqref{temp}, we get
\[
\begin{aligned}
d_{j+1}^N
 \leq  \frac 1{ 1-\Delta t \frac {\delta +2}2 (1-p^2-q^2)}
\left(d_{j}^N + \Delta t (p^{2+\delta} + q^{2+\delta} -1)d_{j}^N+
 \Delta t(p^\delta q^2 +q^\delta p^2) \| |v|^{\delta} \phi_{j}^N\|_{L^1}\right ).
\end{aligned}
\]
By H\"older inequality and the conservation of the mass, we obtain
\[
\| |v|^{\delta} \phi_{j}^N\|_{L^1} \leq \| |v|^{2}
\phi_{j}^N\|_{L^1}^{\frac \delta {2}} \| \phi_j^N\|_{L^1}^{\frac
{2-\delta}2} =1
\]
and so we get the recursive estimate
\[
d_{j+1}^N \leq  \frac 1{ 1-\Delta t \frac {\delta +2}2
(1-p^2-q^2)} \left(d_{j}^N +\Delta t (p^{2+\delta} + q^{2+\delta}
-1)d_{j}^N+
 \Delta t(p^\delta q^2 +q^\delta p^2) \right ).
 \]
Remembering that $\Delta t = \frac TN$, we would like to neglect
the low order terms. We have
\[
\begin{aligned}
& \frac 1{ 1-\Delta t \frac {\delta +2}2 (1-p^2-q^2)}
\left(d_{j}^N +\Delta t (p^{2+\delta} + q^{2+\delta} -1)d_{j}^N+
 \Delta t(p^\delta q^2 +q^\delta p^2) \right )\\
&=\frac{1+\Delta t \frac {\delta +2}2 (1-p^2-q^2)}{ 1-(\Delta t)^2
\left(\frac {\delta +2}2 (1-p^2-q^2)\right )^2} \left(d_{j}^N
+\Delta t (p^{2+\delta} + q^{2+\delta} -1)d_{j}^N+
 \Delta t(p^\delta q^2 +q^\delta p^2) \right )
\\
 &= \frac { d_{j}^N +\Delta t
(p^{2+\delta} + q^{2+\delta} -1 +\frac {\delta +2}2
(1-p^2-q^2))d_{j}^N+ \Delta t(p^\delta q^2 +q^\delta p^2)
 + (\Delta t)^2 K\, d_j^N +
(\Delta t)^2  H\, } {1-(\Delta t)^2 \left(\frac {\delta +2}2
(1-p^2-q^2)\right )^2},
 \end{aligned}
\]
where $H$, $K$ are positive constants, depending only on $p,q$ and
$\delta$. Denoting
\[
\SSS(\delta)= p^{2+\delta} + q^{2+\delta} -1 +\frac {2+\delta}2
(1-p^2-q^2), \qquad B_{p,q}(\delta) = p^\delta q^2 +q^\delta p^2,
\]
we got
\[
d_{j+1}^N\leq \frac {d_j^N +  \Delta t \SSS(\delta) d_j^N + \Delta
t B_{p,q}(\delta)  + (\Delta t)^2 H\, d_j^N + (\Delta t)^2 K\, }
{1-(\Delta t)^2 \left(\frac {\delta +2}2 (1-p^2-q^2)\right )^2}
\]
where, by assumption, we have  $\SSS(\delta) <0$. Moreover,
\[
\frac 1{1-(\Delta t)^2 \left(\frac {\delta +2}2 (1-p^2-q^2)\right
)^2} = 1 +  o(\Delta t), \quad \Delta t \to 0,
\]
so we get, for $N\in \N$ large enough, $j=0,\dots, N-1$:
\[
d_{j+1}^N\leq d_j^N - \frac 12\Delta t |\SSS(\delta)|  d_j^N +
2\Delta t B_{p,q}(\delta).
\]
This recursive relation implies
\[
\int_\R |v|^{2+\delta} \, \phi_{j+1}^N(v)\, \d v \leq C_\delta,
\quad j=0,\dots, N-1,
\]
for any $N\in \N$ large enough.

\hfill $\scriptstyle\square$

\noindent{\bf Proof of Proposition \ref{prop-converg}:} We divide
it into several steps.

\bigskip
 \noindent {\bf I STEP: Existence of the limit $g^*(\xi,t)$ of a subsequence.}
 Thanks to inequalities (\ref{mom}) and to
 the definition of $g^N$ we
have therefore
\begin{equation}\label{congn}
\sup_{[0,T]\times\R}\left|g^N(\xi,t)\right|\leq 1,\quad
 \sup_{[0,T]\times\R}\left|\partial_\xi g^N(\xi,t)\right|\leq C,\quad
\sup_{[0,T]\times\R}\left|\partial^2_\xi g^N(\xi,t)\right|\leq
1,\quad
\end{equation}
where $C$ is the same constant as in (\ref{mom}) and $N$ is large
enough. Moreover since $g^N(\xi,t)$ satisfies
\begin{equation}\label{conxi}
\partial_t^-g^N(\xi,t)=\frac 1r\xi\frac{\d}{\d \xi}\hat\phi^N_{K_N}(\xi)+
\hat\phi^N_{K_N-1}(p\xi)\hat\phi^N_{K_N-1}(q\xi)-\hat\phi^N_{K_N-1}(\xi)
\end{equation}
then for any compact set $K\subset\R$ there exists a constant
$C>0$ such that
\begin{equation}\label{conti}
\sup_{[0,T]\times K}\left|\partial_t^- g^N(\xi,t)\right|\leq C,
\quad N\geq \frac{2T}{r}.
\end{equation}
For any compact $K\subset \R$ the function $g^N(\xi,t)$ belongs to
${\cal C}([0,T]\times K)$ and thanks to properties (\ref{congn})
and (\ref{conti}) the sequence is equibounded and equicontinuous.
Therefore by Ascoli-Arzel\`a theorem and by taking the diagonal,
there exists a subsequence $\left\{g^{N_l}(\xi,t)\right \}$ which
converges uniformly  on $[0,T]\times K$ for any compact $K\subset
\R$. Let us call $g^*(\xi,t) $ the limit function. Since
$g^{N_l}(\xi,0)=\hat g_0(\xi)$ then $g^{*}(\xi,0)=\hat g_0(\xi)$.
Moreover for any  $t\in [0,T]$ the function $g^*(\xi,t)\in {\cal
C}^1(\R)$: indeed by Ascoli-Arzel\`a theorem and the diagonal
argument applied now to the sequence $\left\{ \partial_\xi
g^{N_l}(\xi, t)\right \}$, for any $t\in [0,T]$  we get a
subsequence that  converges uniformly to $\partial_\xi g^*(\xi,t)$
on any compact set $K\subset \R$. Since the limit function is
$\partial_\xi g^*(\xi,t)$,  the convergence holds for the original
sequence $\left\{ \partial_\xi g^{N_l}(\xi,t)\right \}$ and it is
not necessary to pass to a subsequence.

In order to get a uniform convergence in both frequency and time,
we remark that by \eqref{congn}  we have that
$\sup_{[0,T]\times\R}\left|\partial^2_\xi g^N(\xi,t)\right|\leq 1$
 and thanks to \eqref{conxi}
and \eqref{mom} we control  $\partial_t\partial_\xi g^N(\xi,t)$
therefore $\partial_\xi g^N(\xi,t)$  is Lipschitz continuous on
$[0,T]\times K$  and uniformely bounded. Again by Ascoli-Arzel\`a
we prove that $\left\{\partial_\xi g^{N_l}(\xi,t)\right \}$
converges uniformly to $\partial_\xi g^*(\xi,t)$ on $[0,T]\times
K$ for any compact set $K\subset \R$. From the uniform convergence
of $g^{N_l}(\xi,t)$ to $g^*(\xi,t)$ and of $\partial_\xi
g^{N_l}(\xi,t)$ to $\partial_\xi g^{*}(\xi,t)$ we get that both
$g^*(\xi,t)$ and  $\partial_\xi g^{*}(\xi,t)$ belong to ${\cal
C}([0,T]\times K)$.
\bigskip

\noindent  {\bf II STEP: $g^*(\xi,t)$ is a solution of equation
(\ref{eqdis})}.
 By a direct computation we obtain:
\begin{equation}\label{eqre}
\partial_t^-g^{N_l}(\xi,t)=\frac
1r\xi\ \partial_\xi
g^{N_l}(\xi,t)+g^{N_l}(p\xi,t)g^{N_l}(q\xi,t)-g^{N_l}(\xi,t)+R_{N_l}(\xi,t)
\end{equation}
where
\begin{equation}
\begin{aligned}
R_{N_l}(\xi,t)&=-\frac 1r\xi\left(\frac{\d}{\d
\xi}\hat\varphi_{K_{N_l}-1}^{N_l}(\xi)-\frac{\d}{\d
\xi}\hat\varphi_{K_{N_l}}^{N_l}(\xi)\right) +\\
&(1-\alpha)\left[\alpha\left(\hat\varphi_{K_{N_l}}^{N_l}(q\xi)-
\hat\varphi_{K_{N_l}-1}^{N_l}(q\xi)\right)\left(\hat\varphi_{K_{N_l}}^{N_l}(p\xi)-
\hat\varphi_{K_{N_l}-1}^{N_l}(p\xi)\right)\right.+\\
&\hat\varphi_{K_{N_l}-1}^{N_l}(p\xi)\left(\hat\varphi_{K_{N_l}-1}^{N_l}(q\xi)
-\hat\varphi_{K_{N_l}}^{N_l}(q\xi)\right)+
\hat\varphi_{K_{N_l}}^{N_l}(q\xi)\left(\hat\varphi_{K_{N_l}-1}^{N_l}(p\xi)
-\hat\varphi_{K_{N_l}}^{N_l}(p\xi)\right)+\\
&(1-\alpha)\left(\hat\varphi_{K_{N_l}}^{N_l}(\xi)
-\hat\varphi_{K_{N_l}-1}^{N_l}(\xi)\right) \Big ].
\end{aligned}
\end{equation}
Since $\left\{R_{N_l}(\xi,t)\right \}$ converges to zero for any
$t\in[0,T]$ uniformely on any compact set $K\subset \R$, the whole
right hand side of equation \eqref{eqre} converges; let us call
$L(\xi,t)$ the limit function, so that
\begin{equation}\label{eqlim}
L(\xi,t)=\frac 1r\xi\partial_\xi
g^{*}(\xi,t)+g^{*}(p\xi,t)g^{*}(q\xi,t)-g^{*}(\xi,t)
\end{equation}
and $\left\{\partial_t^- g^{N_l}(\xi,t)\right \}$ converges to
$L(\xi,t)$ for any $t\in[0,T]$ uniformely on any compact set
$K\subset \R$. Thanks to  property (\ref{conti}) and to Lebesgue's
dominated convergence theorem we remark that
$\left\{\partial_t^-g^{N_l}(\xi,t)\right \}$ converges to
$L(\xi,t)$ in ${\cal D}'(]0,T[\times K^\circ)$.  Since
$\left\{g^{N_l}(\xi,t)\right \}$ converges to $g^{*}(\xi,t)$
uniformly on $[0,T]\times K$ and so in ${\cal D}'(]0,T[\times
K^\circ)$ we have that $\left\{\partial_t g^{N_l}(\xi,t)\right \}$
converges in distributions  to $\partial_t g^{*}(\xi,t)$. For the
uniqueness of the limit and the fact that in distributions
$\partial_t^- g^{N_l}(\xi,t)=\partial_t g^{N_l}(\xi,t)$, we obtain
 $\partial_t g^{*}(\xi,t)=L(\xi,t)$ in the sense of distributions.
Finally  since $g^*(\xi,t)$ and  $\partial_\xi g^{*}(\xi,t)$
belongs to ${\cal C}([0,T]\times K)$, for any compact $K$  the
right hand side of \eqref{eqlim} is continuous in both variables
 therefore the same is true for $\partial_t g^{*}(\xi,t)$. This
 implies that $g^{*}(\xi,t)$ belongs to
 ${\cal C}'([0,T]\times \R)$.

\bigskip

 \noindent {\bf III STEP: we show that $g^*(\xi,t)=\hat
g(\xi,t)$}.
 Defining $F^*(\xi,t)=g^*(\xi\sqrt{E(t)},t)$ and
 $\hat f(\xi,t)=\hat g(\xi\sqrt{E(t)},t)$ we obtain two solutions of
  equation (\ref{cauchy}). Both $F^*(\xi,t)$ and $\hat f(\xi,t)$
 are ${\cal C}'([0,T]\times \R)$
 and bounded by 1. Denoting
 $h(\xi,t)=\hat f(\xi,t)-F^*(\xi,t)$ by an  easy computation we
 are led to
 $$
| \partial_th+h|\leq 2\|h(t)\|_\infty.
$$
By Gronwall Lemma we have
$$
\|h(t)\|_\infty\leq {\rm e}^t\|h(0)\|_\infty.
$$
and since $\hat f(\xi,0)=F^*(\xi,0)$  this implies the desired
equality $g^*(\xi,t)=\hat g(\xi,t)$.

\hfill $\scriptstyle\square$


\end{document}